\documentclass[12pt]{amsart}
\usepackage{amssymb,amsthm,amsmath,amstext}
\usepackage{mathrsfs} 

\usepackage[left=1.28in,top=.8in,right=1.28in,bottom=.8in]{geometry}

\usepackage[all,cmtip]{xy}
\usepackage{enumitem}

\usepackage{romannum}

\usepackage{colonequals}

\usepackage{graphicx}

\theoremstyle{plain}
\newtheorem{theorem}{Theorem}[section]
\newtheorem{prop}[theorem]{Proposition}
\newtheorem{lemma}[theorem]{Lemma}
\newtheorem{notation}[theorem]{Notation}

\newtheorem{observation}{Observation}

\newtheorem{cor}[theorem]{Corollary}

\newtheorem{defin}[theorem]{Definition}

\theoremstyle{definition}

\newtheorem{remark}[theorem]{Remark}

\newtheorem{cond}[theorem]{Condition} 

\usepackage{hyperref}
\usepackage{color}

\newcommand{\PP}{\mathcal{P}}
\newcommand{\XX}{\mathscr{X}}

 \DeclareFontFamily{U}{wncy}{}
 \DeclareFontShape{U}{wncy}{m}{n}{<->wncyr10}{}
 \DeclareSymbolFont{mcy}{U}{wncy}{m}{n}
 \DeclareMathSymbol{\Sh}{\mathord}{mcy}{"58}

\hypersetup{colorlinks=true,linkcolor=blue,anchorcolor=blue,citecolor=blue,letterpaper=true}

\makeindex

\begin{document}

    \title[Local-global principles]{Local-global principles for multinorm tori over semi-global fields}

	\author{Sumit Chandra Mishra}

	\address{Indian Institute of Science Education and Research Mohali, Knowledge City, Sector 81, Mohali 140 306, India}

	\email{sumitcmishra@gmail.com}

	\thanks{The author would like to acknowledge the support of IIT Bombay, India for the institute postdoctoral fellowship during this work. The author is currently an institute postdoctoral fellow at IISER Mohali, India.}
	\date{25-05-2023}
	

	\begin{abstract}
		Let $K$ be a complete discretely valued field with
the residue field $\kappa$. 
Assume that cohomological dimension of $\kappa$ is less than or equal to  $1$ (for example, $\kappa$ is an algebraically closed field  or a finite field). 
Let $F$ be the function field of a curve over $K$.
Let $n$ be a squarefree positive integer not divisible by char$(\kappa)$.
Then for any two degree $n$ abelian extensions,
we prove that the local-global principle holds for
the associated multinorm torus with respect to discrete valuations. 
Let $\XX$ be a regular proper model of $F$ such that
the reduced special fibre $X$ is a union of regular curves
with normal crossings. Suppose that $\kappa$ is algebraically closed with $char(\kappa)\neq 2$. If the graph associated to
$\XX$ is a tree (e.g. $F = K(t)$) then we show that the same local-global principle holds for the multinorm torus
associated to finitely many abelian extensions where
one of the extensions is quadratic and others are of degree not divisible by $4$.

	\end{abstract}

	\maketitle

\section{Introduction}

Let $F$ be a field and $\Omega_F$ be
the set of all discrete valuations on $F$.
For $\nu \in \Omega_F,$ let $F_\nu$ denote the completion of $F$ at $\nu$.
For a linear algebraic group $G$ over $F$,
one says that the {\it local-global principle} holds
for $G$ if
for any $G$-torsor $X$,
$X$ has a rational point over $F$
if it has a rational point over
$F_\nu$ for all $\nu \in \Omega_F$.
In case of number fields, we also consider
the completions at archimedean places
while discussing local-global
principles for algebraic groups.
The local-global principle for linear algebraic groups does not always hold.
Let $F$ be a semi-global field, i.e., function field of a curve over a complete
discretely valued field.
The first example of a linear algebraic group $G$ over a semi-global field
for which such a local-global principle fails was that of a multinorm torus given by
Colliot-Th\'el\`ene, Parimala and Suresh
(\cite[Section 3.1. \& Proposition 5.9.]{colliot2016lois}).

Let $L/F$ be an \'etale algebra. Then $\displaystyle L =  \prod_{i=1}^{m} L_i$, where $L_i/F$ are finite separable field extensions for $1 \leq i \leq m$.  We say that an element $\lambda \in F^{\times}$ is a norm from $L$ to $F$ if $\lambda$ is a product of norms from the extensions $L_i$ to $F$. Let  $T_{L/F}$ to denote the
multinorm torus associated to the extensions $L_i/F$. Saying that the local-global principle holds for the multinorm torus $T_{L/F}$ is equivalent to the statement that for every $\lambda \in F^{\times}$, $\lambda $ is a norm from $L$ to $F$ if and only if $\lambda $ is a norm from $L\otimes_F F_{\nu}$ to $F_{\nu}$ for all $\nu \in F_{\nu}$. 
In this article, we study local-global principles for multinorm tori
over semi-global fields.\\

In case of number fields, this has been
studied quite extensively, in particular by
Hürlimann (\cite{hurlimann1984algebraic}),
Colliot-Thélène and Sansuc (unpublished),
Platonov and Rapinchuk (\cite[Section 6.3.]{platonov1993algebraic}),
Prasad and Rapinchuk (\cite[Section 4]{prasad2008local}),
Pollio and Rapinchuk (\cite{pollio2012multinorm}),
Demarche and Wei (\cite{demarche2014hasse}),
Pollio (\cite{pollio2014multinorm}),
Bayer-Fluckiger, Lee and Parimala(\cite{bayer2019hasse}).
Multinorm tori are natural generalizations of norm one tori.
The local-global principle for norm one tori has been studied in \cite{mishra2021local}.\\

Let $K$  be a complete discretely valued field with residue field $\kappa$. 
Let $F$ be the function field of a curve over $K$. 
Let $p$ be a prime not equal to char$(\kappa)$ 
and $L/F$ be a degree $p$ cyclic extension. 
It follows from \cite[Theorem 4.3(ii)]{colliot2012patching} and \cite[Corollary 4.7.5]{philippe2006central} that local global principle holds for norm one torus associated to $L/F$. 
In this article, we have considered the case of multinorm tori 
associated to two separable extensions of degree $p$, and to a finite number of quadratic extensions. 
We have also obtained some results for more general multinorm tori.\\

We say that a field $\kappa$ satisfies Condition \ref{A} with respect to a prime $p$ if for all finite extensions $\kappa'/\kappa$ and any two degree $p$ cyclic extensions $l_1,l_2$ of $\kappa'$, $T_{l_1\times l_2/\kappa'}(\kappa')/R = \{1\}$. We note that if $\kappa$ is a field of cohomological dimension $\leq 1$, for example an algebraically closed field or a finite field, and of characteristic different from $p$ then $\kappa$ satisfies Condition \ref{A} with respect to $p$ (See Remark \ref{cd_leq_one}).\\

We prove the following results$\colon$

\begin{theorem}[Theorem \ref{lgp_two_extenions_of_same_degree_squarefree_degree}]

Let $K$ be a complete discretely valued field
with residue field $\kappa.$
Let $F$ be the function field of a curve over $K$.
Let $L_1$ and $L_2$ be two abelian extensions of $F$ of degree $n$,
where $n$ is a squarefree integer not divisible by $\text{char}(\kappa)$. Let $L= L_1 \times L_2.$
If $\kappa$ satisfies Condition \ref{A} with respect to all $p$ dividing $n$, then the local-global principle holds for the multinorm torus $T_{L/F}$ with respect to discrete valuations.

\end{theorem}

\begin{theorem}[Theorem \ref{lgp_multinorm_p_general_case}]

Let $K$ be a complete discretely valued field
with residue field $\kappa.$
Let $F$ be the function field of a curve over $K$.
Let $L_1$ be a degree $p$ separable extension of $F$.
Let $L_2$ be a Galois extension of $F$ of degree $n$ with Galois group $G$.
Assume that $p^2$ does not divide $n$.
If $p$ divides $n$ then also assume that $G$ has an index $p$ subgroup $H$.
Assume that  $\text{char}(\kappa)$ does not divide $pn$.
Let $L= L_1 \times L_2.$
If $\kappa$ satisfies Condition \ref{A} with respect to $p$, then the local-global principle holds for $T_{L/F}$ with respect to discrete valuations.

\end{theorem}

Let $\XX$ be a regular proper model of $F$
such that the reduced special fibre $X$ is a union of
regular curves with normal crossings. 
Now we consider product of norms from more than two extensions. 
There exist examples of three quadratic extensions over
semi-global fields such that the associated multinorm torus does not satisfy
local-global principle (please see \cite{colliot2016lois}). 
However these examples were
over semi-global field $F$ with graph associated to $\XX$ not a tree.
Assuming that the graph associated to $\XX$ is a tree,
we prove the following$\colon$

\begin{theorem}[Theorem \ref{Sha_dvr_vanishes_for_multinorm}]
Let $K$ be a complete discretely valued field
with residue field $\kappa$ algebraically closed.
Let $F$ be the function field of a curve over $K$. 
Let $\XX$ be a regular proper model of $F$
such that the reduced special fibre $X$ is a union of
regular curves with normal crossings. 
Let $L_i/F, 1\leq i \leq m,$ be
quadratic extensions.
Assume that $\text{char}(\kappa) \neq 2$.
Let $\displaystyle L = \prod_{i=1}^{m} L_i$.
If the graph associated to $\XX$ is a tree
then the local-global principle holds for the multinorm torus $T_{L/F}$
with respect to discrete valuations.

\end{theorem}

As a corollary, we obtain the following result$\colon$

\begin{cor}[Theorem \ref{Sha_dvr_vanishes_for_multinorm_quadratic_general}]
Let $K$ be a complete discretely  valued field with residue field $\kappa$.
Assume that $\kappa$ is algebraically closed. 
Let $F$ be function field of a curve over $K$. 
Let $\XX$ be a regular proper model of $F$
such the reduced special fibre $X$ is a union of
regular curves with normal crossings.
Let $L_0$ be a quadratic extension of $F$.
For $i= 1, 2, \dots, m$, let $L_i$ be a Galois extension of
$F$ of degree $n_i$ with Galois group $G_i$.
Assume that for $i$, $1\leq i \leq m$, $4$ does not divide $n_i$.
Assume that $2$ and $n_i$ are not divisible by \text{char}$(\kappa)$.
Let $\displaystyle L = \prod^{m}_{i=0} L_i.$
If the graph associated to $\XX$ is a tree then the local-global principle holds for $T_{L/F}.$

\end{cor}

As in \cite{mishra2021local}, we use the field patching methods developed by
Harbater, Hartmann and Krashen (\cite{harbater2015local}).
We now mention the plan of this article.
In Section \ref{Prelims}, we mention basic facts and prove lemmas about $R$-equivalence for
multinorm tori.
In Section \ref{Semi-global fields and Patching}, we discuss the semi-global fields, their overfields and their models and their associated graph. We recall some theorems from \cite{harbater2015local} and \cite{colliot2020local} which will be used in the proofs.
In Section \ref{Sha}, we study the relationship between the two obstructions, one coming from all discrete
valuations and one coming from the field patching setup. We show that just as in the norm one tori case, our problem
reduces to studying the obstructions coming from the patching set up. Finally, in the Section \ref{LGP_1} and Section \ref{LGP_2}, we the prove
the local-global results as discussed above.

\section{Preliminaries}
\label{Prelims}

 Let $F$ be a field and $\mathbb{G}_m$ be the multiplicative group. 
 For a finite extension $E/F$, let $R_{E/F}\mathbb{G}_m$ denote the Weil restriction of $\mathbb{G}_m$ (considered as an algebraic group over $E$). 
 Let $L_i,$ $1 \leq i \leq m$, be finite separable extensions of $F.$ 
 Let $\displaystyle L= \prod^m_{i=1} L_i$. 
We define the multinorm torus associated to the extensions $L_i/F$ (or associated to $L/F$), denoted by $T_{L/F}$, as the kernel of the map $$\displaystyle \prod^{m}_{i=1} R_{L_i/F} \mathbb{G}_m \xrightarrow{\displaystyle \prod^{m}_{i=1}N_{L_i/F}} \mathbb{G}_m,$$ 
where $N_{L_i/F}$ is the map induced by the usual norm map $N_{L_i/F}: L^{\times}_i \rightarrow F^{\times}$ for $1 \leq i \leq m$.\\

Thus we have the following short exact sequence:\\
$$\displaystyle 
1 \rightarrow 
T_{L/F} \rightarrow 
\left(\prod_{i=1}^{m}R_{L_i/F} \mathbb{G}_m \right) 
\xrightarrow{\displaystyle \prod_{i=1}^{m}N_{L_i/F}} 
\mathbb{G}_m 
\rightarrow 1.$$

By taking Galois cohomology, we get$\colon$\\

$$\displaystyle 
1 \rightarrow 
T_{L/F}(F) \rightarrow 
\left(\prod_{i=1}^{m}(L^{\times}_i)  \right) 
\xrightarrow{\displaystyle \prod_{i=1}^{m}N_{L_i/F}} 
F^{\times} 
\rightarrow H^1(F,T_{L/F}) \rightarrow 1$$
since $H^1(F,R_{L/F}\mathbb{G}_m) = \{ 1\}$ 
(by Hilbert 90 and Shapiro's lemma).

Hence $\displaystyle H^1(F,T_{L/F}) 
\simeq  F^{\times}/ \left( \prod_{i=1}^{m}N_{L_i/F}(L^{\times}_i)\right).$
Thus, the isomorphism classes of $T_{L/F}$-torsors 
correspond to the equivalence classes of 
$\displaystyle F^{\times}/\left( \prod_{i=1}^{m}N_{L_i/F}(L^{\times}_i)\right),$ with 
the trivial torsor corresponding to 
the equivalence class of $1$. 
\\

The rational points of $T_{L/F}$ over $F$ are given by 
$$T_{L/F}(F)= \{ (a_1, \dots , a_m) \in \prod^{m}_{i=1} L_i^{\times}\; |\; \prod^{m}_{i=1} N_{L_i/F}(a_i) =1 \}.$$

If $L/F$ is a finite separable extension, then $T_{L/F}$ is just the norm one torus associated to $L/F$.\\

\subsection{R-equivalence}

\label{Requivalence}

We refer the reader to \cite{colliot1977r}
and \cite{colliot1987principal} for more details
about $R$-equivalence.

Let $X$ be a variety over a field $F.$
 For a field extension $L$ of $F,$
 let $X(L)$ be the set of $L$-points of $X.$
 We say that two points $x_0, x_1 \in X(L)$ are
 {\it elementary $R$-equivalent}, denoted by
 $x_0 \sim x_1,$ if there is a rational map
 $f \colon \mathbb{P}^1(L) \dashrightarrow X(L)$
 such that $f(0) = x_0$ and $f(1) = x_1.$
 The equivalence relation generated by $\sim$
  is called {\it $R$-equivalence}.
 When $X = G$ is an algebraic group defined
 over $F$ with the identity element $e,$
 we define $RG(L) = $ $\{$ $ x \in G(L) \mid$
 $x$ is $R$-equivalent to $e \}.$
 The elements of $RG(L)$ are called
 $R$-trivial elements. It is well-known
 that $RG(L)$ is a normal subgroup of
 $G(L)$ (cf. \cite[p-1]{gille2010lectures}). 
 Sometimes, we denote
 $G(L)/RG(L)$ by $G(L)/R$.
 Let $L/F$ be a Galois extension
 with Galois group $G,$
 and $T_{L/F}$ be the norm $1$
 torus associated to the extension $L/F.$
 Then for any extension $N/F,$  $RT_{L/F}(N)$
 is the subgroup generated by the set
 $\{ a^{-1}\sigma (a) \mid
 a \in (L\otimes_F N)^{\times}, \sigma \in G \}$
(\cite[Proposition 15]{colliot1977r}). \\

We note a well-known fact here:\\

\begin{lemma}
\label{R_trivial_elements_cyclic_norm_one_tori}
Let $F$ be a field and $L/F$ be a
finite cyclic extension.
Then $T_{L/F}(F)=RT_{L/F}(F).$

\end{lemma}

\begin{proof}
Please see \cite[Section 2, p-2]{gille2010lectures}.

\end{proof}

Now we discuss few basic results about
$R$-equivalence on norm one tori
and multinorm tori. We start with a generalization of \cite[Lemma 1.1(i)]{bayer2019hasse}$\colon$ \\

\begin{prop}
\label{only_nonisomorphic_fields_matter_multinorm_mod_R}
Let $F$ be a field and let
$L_1,L_2, \dots, L_m$ be finite
separable extensions of $F$.
Let $\displaystyle L=\prod_{i=1}^{m} L_i$.
Let $r_{i}$ be positive integers for
$i$, $1 \leq i \leq m.$ Let
$\displaystyle L' = \prod_{i=1}^{m}
L_i^{r_i}$.
Then $T_{L'/F}(F)/R  \simeq T_{L/F}(F)/R.$

\end{prop}

\begin{proof} Let $\phi$ be the automorphism of $\displaystyle \prod^{m}_{i=1} (R_{L_i/F}\mathbb{G}_m)^{r_i}$ which is determined by mapping $$((a_{1,1}, a_{1,2}, \dots , a_{1,r_1-1},  a_{1,r_1}), \dots , (a_{m,1},\dots, a_{m,r_m-1}, a_{m,r_m}))$$ to $$
((a_{1,1}, a_{1,2}, \dots , a_{1,r_1-1},  \prod^{r_1}_{i=1} a_{1,i}),\dots, (a_{m,1},\dots, a_{m,r_m-1}, \prod^{r_m}_{i=1}a_{m,i})).$$ In the above map, it is understood that
if $r_j=1$ for some $j$ then $$(a_{j,1},\dots, a_{j,r_j-1}, \prod^{r_j}_{i=1}a_{j,i}) = (a_{j,1}).$$

Consider the following commutative diagram with exact rows: \\

\xymatrixcolsep{4.5pc}\xymatrix{1
\ar[r] 
&T_{L'/F} 
\ar[d]^{\phi|_{T_{L'/F}}} 
\ar[r] 
&\displaystyle \prod^{m}_{i=1} (R_{L_i/F} \mathbb{G}_m)^{r_i} 
\ar[d]^{\phi} 
\ar[r]^{\mu_1} 
&\mathbb{G}_m 
\ar[d]^{||} 
\ar[r]
&1 \\
1 
\ar[r] 
&\displaystyle \prod^{m}_{i=1} (R_{L_i/F} \mathbb{G}_m)^{r_i-1} \times T_{L/F} 
\ar[r] 
&\displaystyle \prod^{m}_{i=1}(R_{L_i/F} \mathbb{G}_m)^{r_i} 
\ar[r]^{\mu_2} 
&\mathbb{G}_m 
\ar[r] 
&1}
where 
the map $\mu_1$ and $\mu_2$ are the maps given by mapping $$((a_{1,1}, a_{1,2}, \dots , a_{1,r_1-1},  a_{1,r_1}), \dots , (a_{m,1},\dots, a_{m,r_m-1}, a_{m,r_m}))$$ to 
$\displaystyle \prod^{m}_{i=1} N_{L_i/F}(a_{i,1}\cdots a_{i,r_i})$ and $\displaystyle \prod^{m}_{i=1} N_{L_i/F}(a_{i,r_i})$, respectively. Then $\phi_{T_{L'/F}}$ is an isomorphism. Since $R_{L_i/F} \mathbb{G}_m$ are rational, $R_{L_i/F}(\mathbb{G}_m)/R = \{1\}$ by \cite[Corollary 1.6.]{gille2010lectures}. Now the result follows by additive property of $R$-equivalence (see \cite[Property(1),p-1]{gille2010lectures}). 
\end{proof}

\begin{lemma}
\label{torus_F_times_separable_algebra_is _R_trivial}
Let $F$ be a field and $L_i, 1\leq i \leq m,$ be
finite separable extensions of $F$ and let
$L = F \times \prod_{i=1}^{m}L_i$.
Then $T_{L/F}(F) = RT_{L/F}(F)$.

\end{lemma}

\begin{proof}

Let us consider the map $\displaystyle f \colon
\prod_{i=1}^{m} R_{L_i/F}\mathbb{G}_m
\to T_{L/F}$ which sends
$(a_1, \dots, a_m)$ to $\displaystyle
( \prod_{i=1}^{m} N_{L_i/F}(a_i)^{-1}, a_1, \dots, a_m)$.
Then $f$ is an isomorphism of algebraic groups.
Since $R_{L_i/F}(\mathbb{G}_m)$ are rational, $T_{L/F}$ is rational. Hence
$T_{L/F}(F) = RT_{L/F}(F)$ by
\cite[p.1-2]{gille2010lectures}.

\end{proof}

\begin{lemma}
\label{nthpower}
 Let $L/F$ be a finite Galois extension
 of degree $n$ and $N/F$ be any field extension.
 If $\alpha \in (L \otimes_F N)^{\times},$
 then $ (N_{L \otimes_F N/N }(\alpha)^{-1}) \alpha^n
 \in RT_{L/F}(N).$
\end{lemma}

\begin{proof}
 Let $G$ be the Galois group of $L/F.$
 Since $\displaystyle N_{L \otimes_F N/N}(\alpha)
 = \prod_{\sigma \in G} \sigma(\alpha),$
 we have
 $$\displaystyle (N_{L \otimes_F N/N}(\alpha)^{-1}) \alpha^n =
 {[\prod_{\sigma \in G} \sigma(\alpha)]}^{-1}{\alpha^n} =
 \prod_{\sigma \in G} \left({[\sigma(\alpha)]}^{-1}{\alpha}\right).$$
 Since
 $[\sigma(\alpha)]^{-1}{\alpha} \in RT_{L/F}(N)$ by \cite[Proposition 15]{colliot1977r}, $(N_{L \otimes_F N/N}(\alpha)^{-1}) \alpha^n
 \in RT_{L/F}(N).$
\end{proof}

\section{Semi-global fields and Patching}
\label{Semi-global fields and Patching}

Let $F$ be a field and $\Omega _{F}$ be the set of all
equivalence classes of discrete valuations $\nu$ on $F.$
For $\nu \in \Omega _{F},$ let $F_{\nu }$ denote
the completion of $F$ at $\nu$ and $\kappa(\nu)$
the residue field at $\nu.$
For an algebraic group $G$ over $F,$ let
 $$\Sh (F, G) \colonequals \mathsf{ker} \left( H^1(F, G) \rightarrow
\displaystyle \prod_{\nu \in \Omega_{F}} H^1(F_{\nu }, G) \right).$$\\

In this article, we are concerned with
a special class of fields called semi-global fields.\\

\begin{defin}
\label{SGF}
A {\it semi-global field} is the function field of
a curve over a complete discretely valued field.

\end{defin}

\textbf{Setting}: Let $T$ be a complete discretely valued ring with
fraction field $K$ and residue field $\kappa$.
Let $t \in T$ be a parameter.
Let $F$ be a function field of a curve over $K.$
Thus $F$ is a semi-global field.
Then there exists a regular $2$-dimensional integral
scheme $\XX$ which is proper over $T$ with function
field $F.$ We call such a scheme $\XX$ a
{\it regular proper model} of $F.$
Further there exists a regular proper model
of $F$ with the reduced special fibre
a union of regular curves with only normal crossings.
Let $\XX$ be a regular proper model of $F$
with the reduced special fibre $X$ a union of regular
curves with only normal crossings.
Such a model always exists by \cite[p-193]{lipman1975introduction}.

\subsection{Overfields of a semi-global field}
\label{prelim_patching_overfields}

For a semi-global field $F$ along with a regular proper model $\XX$ as in the previous paragraph, one can
associate three different kinds of overfields of $F$.
We describe them below and discuss how they are
related to each other. \\

 For any point $x$ of $\XX$,
 let $R_x$ be the local ring at $x$ on $\XX,$
 $\hat{R}_x$ the completion of the local ring $R_x,$
 $F_x$ the fraction field of $\hat{R}_x$ and
 $\kappa(x)$ the residue field at $x.$ \\

For any non-empty open subset $U$ of $X$ that is properly contained
in an irreducible component of $X$,
let $R_U$ be the subring of $F$
consisting of the rational functions which are
regular at every point of $U.$
Let $\hat{R_U}$ be the $t$-adic completion
of $R_U$ and $F_U$ the fraction field of $\hat{R_U}.$\\

Let $\eta \in X$ be a codimension zero point and $P \in X$
a closed point such that $P$ is in $X_{\eta}$ where $X_\eta$ 
 is the closure of $\eta$. 
Such a pair $(P,\eta )$ is called a {\it branch}.
For a branch $(P,\eta ),$ we define $F_{P,\eta }$
to be the completion of $F_P$ at the discrete valuation
of $F_P$ associated to $\eta .$
We call such fields {\it branch fields}.
If $\eta$ is a codimension zero point of $X$,
$U\subset X_\eta$ an open subset of $X$ as above and
$P \in X_\eta$ a closed point, then
we will use $(P,U)$ to denote the
branch $(P,\eta)$ and $F_{P,U}$ to denote
the field $F_{P,\eta}.$ \\

With $P, U, \eta $ as above,
there are natural inclusions of
$F_P,$ $F_U$ and $F_{\eta }$
into $F_{P,\eta }=F_{P,U}.$
Also, there is a natural inclusion
of $F_{U}$ into $F_{\eta }.$ \\

Let $\PP$ be a nonempty finite set of closed points of $X$
that contains all the closed points of $X$,
where distinct irreducible components of $X$ meet.
Let $\mathcal{U}$ be the set of connected components
of the complement of $\PP$ in $X$ and let $\mathcal{B}$
be the set of branches $(P, U)$ with
$P \in \PP$ and $U \in \mathcal{U}$
with $P$ in the closure of $U$.

\subsection{The associated graph}

 We have a graph $\Gamma(\XX,\PP)$ associated to 
 $\XX$ and $\PP$ whose 
 vertices are elements of $\PP \cup \mathcal{U}$ 
 and there is an edge connecting $P \in \PP$ and $U\in \mathcal{U}$ if $(P,U)$ is a branch. 
There are no edges between any vertices 
which are in $\PP$ (resp. $\mathcal{U}$). Thus
$\Gamma(\XX, \PP)$ is a finite bipartite  graph with 
parts $\PP $ and $\mathcal{U}$. 
If $\PP'$ is another non-empty finite set of closed points of $X$ 
containing all the closed points of 
$X$ where  distinct irreducible components 
of $X$ meet, then $\Gamma(\XX,\PP)$ is a tree is and only if 
$\Gamma(\XX, \PP')$ is a tree (\cite[Remark 6.1(b)]{harbater2015local}). 
Hence if $\Gamma(\XX,\PP)$ is a tree for some $\PP$ as above, 
then we say that  the graph associated to $\XX$ is a tree.\\

We start with a basic fact about finite bipartite trees.

\begin{lemma}
\label{factorization_in_graph_multinorm}

Let $\Gamma$ be a finite bipartite tree. Let
$V$ be the set of vertices with parts $V_1$
and $V_2$, and $E$ be the set of edges.
Suppose that for every
edge $e \in E,$ we have an abstract group
$G_e$ and for every vertex $v \in V$, we have
an abstract group $G_v$ such that
for all $v \in V$ and edges $e$ with
$v$ as one of the vertex, we have a
surjective group homomorphism
$f_{v,e} \colon G_{v} \rightarrow G_{e}$.
Then for every tuple $\displaystyle (g_{e})_{e \in E} \in \prod_{e \in E} G_e$,
there exists a tuple $\displaystyle (g_v)_{v\in V} \in \prod_{v \in V} G_v$ such that
if $e$ is the edge joining the vertices $v_1\in V_1$ and $v_2\in V_2$,
$$ g_{e} = f_{v_1,e}(g_{v_1}) \cdot f_{v_2,e}(g_{v_2}).$$
\end{lemma}

\begin{proof}

Without loss of generality,
we may assume that $\Gamma$
is a connected tree.
We prove the lemma by induction
on the number of vertices of $\Gamma$.
Suppose that $\Gamma$ has only one vertex.
Then there is nothing to prove. \\

Suppose that $\Gamma$ has $n$
vertices, where $n > 1$. Let $\displaystyle
(g_e)_{e \in E} \in \prod_{e \in E} G_e.$
Since $\Gamma$ is a finite connected tree, there
exists a vertex $v_0$ with exactly one edge
$\theta$ at $v_0$.
Without loss of generality, we may assume that
$v_0 \in V_1$. Let $\Gamma'$ denote the graph
obtained by deleting the vertex $v_0$ and the edge
$\theta$. Then $\Gamma'$ is again a finite bipartite
tree with $n-1$ vertices. Thus, by induction hypothesis,
there exists a tuple $\displaystyle (g_v)_{v \in V \setminus \{v_0\}}
\in \prod_{v \in V\setminus \{v_0\}} G_{v}$ such that
whenever $e \in E\setminus \{\theta \}$ is an edge
between some vertices $v_1$ and $v_2$ in $V \setminus \{v_0\}$,
we have $g_{e} = f_{v_1,e}(g_{v_1}) \cdot f_{v_2,e}(g_{v_2}).$\\

Let $v'_{0} \in V_2$ be the other vertex
of the edge $\theta$. Choosing $g_{v_0}$
to be an element in $f^{-1}_{v_0,\theta}
(g_{\theta} \cdot f_{v'_{0},\theta}(g_{v'_0})^{-1})$,
we see that $(g_{v})_{v \in V}$ satisfies the
required property. Hence we are done.

\end{proof}

\subsection{Tate-Shafarevich groups}

Let $G$ be a linear algebraic group over $F.$ Let us define
 $$\displaystyle \Sh _{\XX, \PP}(F, G) \colonequals
 \mathsf{ker}\left( H^1(F,G) \rightarrow
 \prod _{\xi \in \PP \cup \mathcal{U}} H^1 (F_{\xi }, G) \right).$$\\

 If $\XX$ is understood, we write 
 $\Sh _{ \PP}(F, G)$ for $\Sh _{\XX, \PP}(F, G)$.\\

 Similarly, let us define $$\displaystyle \Sh _{\XX, X}(F, G) \colonequals
 \mathsf{ker}\left( H^1(F,G) \rightarrow
 \prod _{P \in X} H^1(F_P,G)\right).$$\\

Again, if $\XX$ is understood,
we write $\Sh _{X}(F,G)$ for $\Sh _{\XX, X}(F, G)$.\\

 We have a bijection(\cite[Corollary 3.6.]{harbater2015local})
 $\colon$ \hfill\\
 \begin{center}
 $\displaystyle \prod_{U\in \mathcal{U}} G(F_U)
 \backslash\prod_{(P,U) \in \mathcal{B}} G(F_{P,U})/
 \prod_{P\in \PP} G(F_P) \rightarrow \Sh _{\PP}(F, G)$
 \end{center}

For an $F$-torus $T$, a new double coset representation of $\Sh _{\PP}(F, T)$ is obtained in \cite{colliot2020local}:

\begin{theorem}{\cite[Theorem 3.1(b)]{colliot2020local}}
\label{Sha_p_in_terms_of_mod_R} There is an isomorphism of abelian groups\\

 $\displaystyle
 \prod_{U\in \mathcal{U}} T(F_U)/R\, \backslash \,
 \prod_{(P,U) \in \mathcal{B}} T(F_{P,U})/R \, /\, \prod_{P\in \PP} T(F_P)/R
 \simeq \Sh _{\PP}(F, T)
.$

\end{theorem}

\vspace{0.05in}

We recall basic results by Harbater, Hartmann and Krashen on the 
relationship between the various Tate-Shafarevich groups:\\

\begin{theorem}{\cite[Corollary 5.9.]{harbater2015local}}
\label{HHK_theorem1}
$\displaystyle \Sh _{X}(F, G) = \bigcup \Sh _{\PP}(F, G)$,
where union ranges over all finite subsets
$\PP$ of closed points of $\XX$
which contain all the singular points of $X.$

\end{theorem}

\begin{theorem}{\cite[Proposition 8.2.]{harbater2015local}}
\label{HHK_theorem2}
$\Sh _{X}(F, G) \subseteq \Sh (F, G)$.
\end{theorem}

As a corollary to Lemma \ref{factorization_in_graph_multinorm}, we get the following:

\begin{cor}
\label{surjective_from_P_and_U_implies_ShaP_trivial}

Let $K$ be a complete discretely valued field
with residue field $\kappa$ and $F$ be
the function field of a curve over $K.$ 
Let $\XX$ be a regular proper model of $F$
with the reduced special fibre $X$ a union of regular
curves with only normal crossings. 
Let $T$ be a torus defined over $F$. Let $\PP$ be a non-empty finite set of closed points of $X$ containing
all the nodal points .

Assume that$\colon$\\

$(i)$ the graph associated to $\XX$
is a tree,\\

$(ii)$ the natural map $T(F_{U})
\rightarrow T(F_{P,U})/R$ is surjective
for all possible $U$ and branches $(P,U)$,
where $U$ is one of
the components of $X\setminus {\PP}$, and \\

$(iii)$ the natural map $T(F_{P})
\rightarrow T(F_{P,U})/R$ is surjective
for all possible $P$ and branches $(P,U)$.\\

Then $\Sh_{\PP}(F,T)= 0.$

\end{cor}

\begin{proof}

By Theorem \ref{Sha_p_in_terms_of_mod_R}, we have an isomorphism of abelian groups

 $$\displaystyle
 \prod_{U\in \mathcal{U}} T(F_U)/R\, \backslash \,
 \prod_{(P,U) \in \mathcal{B}} T(F_{P,U})/R \, /\, \prod_{P\in \PP} T(F_P)/R
 \simeq \Sh _{\PP}(F, T)
.$$
Note that we have a finite bipartite tree $\Gamma$ with
$V_1 = \PP$, $V_2 = \mathcal{U}$
and  $E= \mathcal{B}$, where
$(P,U) \in \mathcal{B}$ is the edge joining
$P \in \PP$ and $U \in V_2$. Now considering the groups
$G_{P,U} = T(F_{P,U})/R$, $G_{P} = T(F_{P})$ and
$G_{U} = T(F_{U})$ and using assumptions $(i)-(iii)$, the result follows
from Lemma \ref{factorization_in_graph_multinorm}.

\end{proof}

\section{$\Sh$ vs $\Sh_X$}
\label{Sha}

We compare the groups $\Sh$ and
$\Sh_X$ for multinorm tori.\\

\begin{notation}

For a regular proper model $\XX$
of a semi-global field $F$ and
a field extension $L/F$, we use
$ram_{\XX}(L/F)$ to denote the
ramification locus for the extension
$L/F$ with respect to $\XX$. Also,
for $\lambda \in F^{\times}$,
we use $supp_{\XX}(\lambda)$
to denote the support of $\lambda$ in $\XX$.
For any field $F$ with a discrete valuation $\nu$, by abuse of notation, we would say that an element $x \in F$ is a unit if $\nu(x)=0$.
\end{notation}

\begin{theorem}
\label{branch_to_P_multinorm}
Let $A$ be a complete regular local ring of dimension
$2$ with the residue field $\kappa$ and the fraction
field $F$. Let $L_1, L_2, \dots, L_m$ be finite Galois
extensions of $F$ with $[L_i \colon F] = n_i$. Assume that $\text{char}(\kappa)$ does not divide $n_i$ for $1 \leq i \leq m$.
Let $\mathfrak{m} = (\pi_1, \pi_2)$ be the maximal
ideal of $A$. Assume that $L_i/F$ are unramified on $A$
except possibly at $\pi_1$, $\pi_2$.
Let $\lambda = u \pi^r_1 \pi^s_2 \in F$, where
$u \in A$ is a unit and $r, s$ are integers.
Suppose that $\lambda$ is a norm
from $\displaystyle \prod^{m}_{i=1} L_i\otimes_F F_{\pi_1}$ to $F_{\pi_1}$.
Then $\lambda$ is a norm from
$\displaystyle \prod^{m}_{i=1} L_i$ to $F$.
\end{theorem}

\begin{proof}

Let us consider $\displaystyle n = \prod_{i=1}^{m} n_i$.
By our assumption, $\text{char}(\kappa)$ does not divide $n$.
Let $\displaystyle \lambda = \prod_{i=1}^{m} \beta_i$,
where $\beta_i\in F_{\pi_1}^{\times}$, and $\beta_i = N_{L_i\otimes_F F_{\pi_1}/F_{\pi_1}}(\alpha_i)$
for some $\alpha_i \in (L_i\otimes_F F_{\pi_1})^{\times}$.
By \cite[Lemma 4.7]{mishra2021local}, we can write
$\beta_i = u_i \pi_1^{r_i} \pi_2^{s_i} {b_i}^{n}$
for some $u_i \in A$ a unit, $b_i \in F_{\pi_1}$ and
integers $r_i, s_i$.

Let $\displaystyle b = \prod_{i=1}^{m} (b_i)^n$. Then
$\displaystyle b = \lambda  \prod_{i=1}^{m} (u_i \pi_1^{r_i} \pi_2^{s_i})^{-1}$.
Hence $b \in F$. Since $b$ is a $n^{\rm{th}}$ power in $F_{\pi_1}$,
it is a $n^{\rm{th}}$ power in $F$ by considering $m=1$ case in \cite[Corollary 5.5.]{parimala2018local} and using the identification $H^1(F,\mu_n) \simeq F^{\times}/F^{\times n}$. Hence $b$ is a norm from $L_i/F$. Thus we can assume that $b_i = 1$ and $\beta_i \in F$.
Since $\beta_i = u_i \pi^{r_i}_1 \pi^{s_i}_2$ is a norm from $L_i\otimes_F F_{\pi_1}/F_{\pi_1}$, by \cite[Theorem 4.9]{mishra2021local},
$\beta_i$ is a norm from $L_i/F$ and hence
$\lambda = \prod \beta_i$ is a norm from $\displaystyle \prod^{m}_{i=1} L_i$ to $F$.

\end{proof}

\begin{theorem}
\label{union_of_Sha_X_equals_Sha_dvr_multinorm}

Let $K$ be a complete discretely valued field with
residue field $\kappa.$
Let $F$ be the function field of a curve over $K$.
Let $L_i/F$ be Galois field extensions of degrees not divisible by $\text{char}(\kappa)$ for $i=1, 2, \dots , m.$
Let $\displaystyle L = \prod_{i=1}^{m} L_i$.
Let $\XX_0$ be a regular proper model of $F$.
Then $$\displaystyle \Sh (F,T_{L/F}) =
\bigcup_{X} \Sh _X (F,T_{L/F}),$$
where $X$ runs over the reduced special fibres
of regular proper models $\XX$ of $F$ which are
obtained as a sequence of blow-ups of
$\XX_0$ centered at closed points.

\end{theorem}

\begin{proof}

Let $\displaystyle x \in \Sh (F,T_{L/F}) \subset H^1(F, T_{L/F}).$
Since $$\displaystyle H^1(F, T_{L/F}) \simeq
F^{\times}/(\prod_{i=1}^{m} N_{L_{i}/F}(L_{i}^{\times})),$$
we may choose a lift $\lambda \in F^{\times}$ of $x$.
By \cite[p-193]{lipman1975introduction},
there exists a  sequence of blow-ups $\XX \to \XX_0$
centered at closed points such that the union of
$supp_{\XX}(\lambda ),$ $ram_{\XX}({L_i}/F)$
and the reduced special fiber $X$ of $\XX$
is a union of regular curves with normal crossings.
We show that $\displaystyle x \in \Sh _{X}(F,T_{L/F}).$\\

Let $P\in X.$
Suppose $P$ is a generic point of $X.$
Then $P$ gives a discrete valuation
$\nu$ of $F$ with $F_\nu = F_P.$
Since $\displaystyle x \in \Sh (F,T_{L/F}),$ $x$ maps to
$0$ in $H^1(F_P, T_{L/F}).$\\

 Suppose that $P$ is a closed point.
 Let $\eta_1$ be the generic point of
 an irreducible component of $X$ containing $P.$
 Let $\mathcal{O}_{\XX,P}$ be the local ring at $P$
 and $\mathfrak{m}_{\XX,P}$ be its maximal ideal.
 Then by our choice of $\XX,$
 $\mathfrak{m}_{\XX,P}= (\pi _1, \pi _2)$, where
 $\pi _1$ is a prime defining $\eta_1$ at $P,$
 $\lambda = u \pi _1^r \pi _2^s$ for some unit
 $u\in \mathcal{O}_{\XX,P}$ and integers $r,s,$ and
 all $L_i/F$ are unramified on $\mathcal{O}_{\XX,P}$
 except possibly at $\pi _1, \pi _2.$
 Since $L_i/F$ are  Galois extensions,
 $\displaystyle {L_i}\otimes _F F_{P} = \prod (L_i)_P$
 for some Galois extensions $(L_i)_P/F_P.$
 Since $L_i/F$ are unramified on
 $\mathcal{O}_{\XX,P}$ except possibly at $\pi _1, \pi _2,$
 $(L_i)_P/F_P$ is unramified on $\hat{\mathcal{O}}_{\XX,P} $
 except possibly at $\pi _1, \pi _2.$
 Since $\lambda $ is a lift of $\displaystyle x \in \Sh (F,T_{L/F}),$
 $\lambda $ is a norm from
 $L_i\otimes _F F_{\eta_1}$ to $F_{\eta_1}.$
 Since $F_{\eta_1} \subset F_{P, \eta_1},$
 $\lambda$ is a norm from
 $L_i\otimes _F F_{P, \eta_1}$ to $F_{P, \eta_1 }.$
 Hence $\lambda $ is a norm from
 $(L_i)_P \otimes F_{P, \eta_1}$ to $F_{P, \eta_1}.$
 Thus, by Theorem \ref{branch_to_P_multinorm},
 $\lambda$ is a norm from $(L_i)_P$ to $F_P$ and
 $x$ maps to $0$ in $H^1(F_P, T_{L/F}).$
 Therefore $\displaystyle x \in \Sh _X (F,T_{L/F}).$
 By \ref{HHK_theorem2}, we have $\displaystyle \Sh (F,T_{L/F}) = \bigcup _{X} \Sh _X (F,T_{L/F}),$
 where $X$ runs over the reduced special fibres
 of regular proper models of $F$ which are
obtained as a sequence of blow-ups of
$\XX_0$ centered at closed points of $\XX_0$.

\end{proof}

\begin{remark}
\label{divisorial_multinorm}
The proof of Theorem \ref{union_of_Sha_X_equals_Sha_dvr_multinorm}
also works if we just consider divisorial discrete valuations
instead of considering all discrete valuations on $F$.
\end{remark}

\section{Local-Global Principles for product of norms from two extensions}
\label{LGP_1}

We will be using the following notation
throughout this section and the next section.

\begin{notation}

For a semi-global field $F$ and
field extensions $L_i/F$, we
use $L_{i,\eta}$, $L_{i,P}$,
$L_{i,P,\eta}$ and $L_{i,P,U}$
to denote
$L_i \otimes_F F_{\eta}$,
$L_i \otimes_F F_P$,
$L_i \otimes_F F_{P,\eta}$, and
$L_i \otimes_F F_{P,U}$ respectively.

\end{notation}

Let $K$ be a complete discretely valued field
with residue field $\kappa.$
Let $F$ be the function field of a curve over $K$ and
let $L_1$ and $L_2$ be two cyclic
extensions of $F$ each of degree $p$
for some prime $p$.
Let $L = L_1 \times L_2$ and $T_{L/F}$ denote the
associated multinorm torus.
Under some assumptions on $\kappa$, for example, assuming $\kappa$
is either algebraically closed or a finite field with $p \neq \text{char}(\kappa)$, we prove that local-global principle holds for such multinorm tori with respect to discrete valuations. We prove this by proving first that for a regular proper model $\XX$ of $F$ with the reduced special fibre $X$ a union of
regular curves with normal crossings, and for any branch $F_{P,\eta}$, we have 
$T_{L/F}(F_{P,\eta})/R = \{1\}$. For two quadratic extensions, we also show that no assumptions on the residue field is required except that $\text{char}(\kappa) \neq 2.$\\

\begin{observation} 
\label{coprime_degrees}
Let $F$ be a semi-global field and 
let $L_i, 1\leq i \leq m,$ be
finite separable extensions of $F$. Let $\displaystyle L=\prod^{m}_{i=1}L_i$. Let $n_i$ denote the degree of the field extension $L_i/F$ for $1\leq i \leq m$. Assume that 
$\mathsf{gcd}(n_1,\dots ,n_m)=1$. 
For any $\lambda \in F^{\times}$, 
$\lambda^{n_i} \in N_{L_i/F}(L^{\times}_i)$ for $1 \leq i \leq m$. 
Since $\mathsf{gcd}(n_1,\dots ,n_m)=1$, this implies that $\lambda $ is a norm from $L$ to $F$. Thus the local-global principle trivially holds for the multinorm torus $T_{L/F}$. 

\end{observation}

We start with a basic result about
multinorm tori over arbitrary fields.

\begin{lemma}
\label{R_trivial_elements}
Let $L_i, 1\leq i \leq m,$ be
finite separable extensions of a given field
$F$ with $[L_i \colon F]=n_i$.
Let $\displaystyle n\colonequals  \mathsf{lcm} (n_i \mid 1 \leq i \leq m).$
Let $\displaystyle L = \prod_{i=1}^{m} L_i$.
Let $\alpha_i \in F^{\times}$ for all $i, 1 \leq i \leq m-1$.
Then the element
$$\displaystyle (\alpha^{n/n_1}_1, \alpha^{n/n_2}_2, \dots,
\alpha^{n/n_{m-1}}_{m-1},
\prod_{i=1}^{m-1}\alpha^{-n/n_{m}}_i)
\in RT_{L/F}(F).$$

\end{lemma}

\begin{proof}

Let us consider $f(t) \in T_{L/F}(F(t))$ given by
$$\displaystyle f(t) =
\left( \left(\frac{t+\alpha_1}{t+1}\right)^{n/{n_1}},
\left( \frac{t+\alpha_2}{t+1}\right)^{n/{n_2}},
\dots , \left(\frac{t+\alpha_{m-1}}{t+1}\right)^{n/{n_{m-1}}},
\prod_{i=1}^{m-1}\left(\frac{t+1}{t+\alpha_i}\right)^{n/{n_{m}}}\right) .$$\\

Then for $t=0$, we get $$\displaystyle f(0) =
(\alpha^{n/n_1}_1, \alpha^{n/n_2}_2, \dots,
\alpha^{n/n_{m-1}}_{m-1},
\prod_{i=1}^{m-1}\alpha^{-n/n_{m}}_i) $$
and for $\displaystyle t= \infty$, we get $f(\infty) = (1,1, \dots, 1)$.\\

Hence $\displaystyle (\alpha^{n/n_1}_1, \alpha^{n/n_2}_2, \dots,
\alpha^{n/n_{m-1}}_{m-1},
\prod_{i=1}^{m-1}\alpha^{-n/n_{m}}_i)$
belongs to $RT_{L/F}(F).$

\end{proof}

Now we study $R$-trivial elements
of $T_{L/F}(F)$ for complete discretely
valued fields $F$ where $L$ is product of
two degree $p$ cyclic extensions.\\

\begin{lemma}
\label{1unramified_1ramified_ext_tori_mod_R_trivial}

Let $F$ be a complete discretely valued field
with residue field $\kappa$.
Let $L_1/F$ be
a degree $p$ unramified cyclic extension and
$L_2/F$ be a degree $p$ ramified cyclic extension.
Assume that $p \neq \text{char}(\kappa).$
Let $L = L_1 \times L_2$. Then
$T_{L/F}(F)/R =\{1\}.$

\end{lemma}

\begin{proof}

By \cite[Lemma 2.4]{parimala2018local}, there exists $\pi$ a parameter in $F$ such that $L_2 = F(\sqrt[p]{\pi})$
. Then
$\pi$ is also a parameter in $L_1$.
Let $(\mu_1, \mu_2) \in T_{L/F}(F)$.
We can write $\mu_1 = u_1 \pi^{q_1}$
and $\mu_2 = u_2  (\sqrt[p]{\pi})^{q_2}$
where $u_1 \in L_1$ and $u_2 \in L_2$ are
units and $q_1, q_2$ are integers. 
Note that $N_{L_2/F}(\sqrt[p]{\pi})=(-1)^{n_p}\pi$ where
$n_p = 1$ if $p=2$ and $n_p = 2$ if $p>2.$ We have $1 = N_{L_1/F}(\mu_1) N_{L_2/F}(\mu_2)
= N_{L_1/F}(u_1) \pi^{pq_1}
N_{L_2/F}(u_2) (-1)^{n_p q_2}\pi^{q_2}$. Then $q_2 = -pq_1$, and $(\mu_1,\mu_2)=(u_1\pi^{q_1}, u_2\pi^{-q_1})$. Also, $(-1)^{n_p q_2} = 1$
for any prime $p$. \\

By Lemma \ref{R_trivial_elements}, we have 
$(\pi^{q_1},\pi^{-q_1}) \in RT_{L/F}(F).$
Thus it is enough to consider the case when
$\mu_1$ and $\mu_2$ are units in $L_1$
and $L_2$ respectively. 
Let $\ell_2$ be the residue field of $L_2$. Since
$L_2/F$ is totally ramified, there exists $\mu_3\in F$ a unit such that $\overline{\mu_3}=\overline{\mu_2} \in \ell_2$. Since $F$ is complete, $L_2$ is also complete. Thus $\mu_2\mu^{-1}_3$ is a $p^{\rm th}$ power in $L_2$. Thus $\mu_2=\mu_3\mu^{p}_{4}$ for some unit $\mu_4 \in L_2$.
Thus we have $1 = N_{L_1/F}(\mu_1)
[N_{L_2/F}(\mu_3)\mu_4]^{p}.$
Let $\alpha = N_{L_2/F}(\mu_3)\mu_4 \in F$.
Then $(\alpha, \alpha^{-1}) \in RT_{L/F}(F)$ by Lemma \ref{R_trivial_elements}.
Hence $(\mu_1 \alpha, \mu^{p}_3\mu_4 \alpha^{-1})$
is $R$-equivalent to $(\mu_1,\mu_2)$.
We have $N_{L_1/F}(\mu_1 \alpha)= 1$ and
$N_{L_2/F}(\mu^{p}_3\mu_4 \alpha^{-1})=1$.
Now, since $L_1/F$ and $L_2/F$ are cyclic extensions,
by Lemma \ref{R_trivial_elements_cyclic_norm_one_tori},
$T_{L_1/F}(F)/R =\{1\}$ and $T_{L_2/F}(F)/R=\{1\}.$
Thus $(\mu_1 \alpha, \mu^{p}_3\mu_4 \alpha^{-1})$
is in $RT_{L/F}(F)$ and we are done.

 \end{proof}

\begin{lemma}
\label{both_ramified_ext_tori_mod_R_trivial}

Let $F$ be a complete discretely valued field
with residue field $\kappa$.
Let $L_1/F$ and $L_2/F$ be
two degree $p$ ramified cyclic extensions.
Assume that $p \neq \text{char}(\kappa).$
Let $L = L_1 \times L_2$. Then
$T_{L/F}(F)/R =\{1\}.$

\end{lemma}

\begin{proof}

By \cite[Lemma 2.4]{parimala2018local}, there exists a parameter $\pi \in F$ such that 
$L_1 = F(\sqrt[p]{\pi})$. Similarly, $L_2 = F(\sqrt[p]{v \pi})$ for some unit $v \in F^{\times}.$
We can assume that $v \notin F^{\times p}$, 
otherwise both extensions are isomorphic
and by Proposition \ref{only_nonisomorphic_fields_matter_multinorm_mod_R},
$T_{L/F}(F)/R \simeq T_{L_1/F}(F)/R.$
By Lemma \ref{R_trivial_elements_cyclic_norm_one_tori},
since $L_1/F$ is cyclic,
we have $T_{L_1/F}(F)/R =  \{1\}$.
Hence we are done.\\

Let $(\mu_1, \mu_2) \in T_{L/F}(F).$
We can write $\mu_1 = u_1 (\sqrt[p]{\pi})^{q_1}$
and $\mu_2 = u_2 (\sqrt[p]{v \pi})^{q_2}$ where
$u_1 \in L_1$ and $u_2 \in L_2$ are units and
$q_1, q_2$ are integers.
Note that $N_{L_1/F}(\sqrt[p]{\pi})=(-1)^{n_p}\pi$ and 
$N_{L_2/F}(\sqrt[p]{v\pi})=(-1)^{n_p}v\pi$, where $n_p=1$ if $p=2$ and $n_p=2$ if $p>2$. 
Thus we have $$1 = N_{L_1/F}(u_1) (-1)^{n_pq_1}\pi^{q_1}
N_{L_2/F}(u_2) (-1)^{n_pq_2}v^{q_2} \pi^{q_2}.$$
This implies $q_2 = -q_1.$
Consequently, $(-1)^{n_pq_1}(-1)^{n_pq_2} = 1.$
Also, since the extensions are totally
ramified, $N_{L_1/F}(u_1), N_{L_2/F}(u_2)
\in F^{\times p}$. Hence $v^{q_2} \in F^{\times p}$
and $p$ divides $q_2$ since $v \notin F^{\times p}.$
Let $q_1 = pq$ for some integer $q$.\\

By Lemma \ref{R_trivial_elements}, we have
$((v\pi)^{q}, (v\pi)^{-q}) \in RT_{L/F}(F).$
Thus it is enough to show that
$(\mu_1 (v\pi)^{-q}, \mu_2 (v\pi)^{q})
= (u_1v^{-q}, u_2) \in RT_{L/F}(F).$ Thus
we can assume that $\mu_1$ and $\mu_2$
are units in $L_1$ and $L_2$ respectively.
Furthermore, since
$L_2/F$ is totally ramified, we have
$\mu_2 = \mu_3^{p}\mu_4$ where
$\mu_3 \in L_2$ is a unit and $\mu_4\in F$
is a unit. Continuing as in the last half of the proof of
Lemma \ref{1unramified_1ramified_ext_tori_mod_R_trivial},
we get that $(\mu_1, \mu_2) \in RT_{L/F}(F).$

\end{proof}

\begin{lemma}
\label{both_unramified_ext_tori_mod_R_trivial}
 Let $F$ be a complete discretely valued field with residue field $\kappa$. Let $L_1/F$ and $L_2/F$ be two unramified cyclic extensions of degree $p$. Assume that $p \neq \text{char}(\kappa)$. Let $L=L_1\times L_2$. For $i=1, 2$, let $l_i$ denote the residue field of $L_i$. Let $l= l_1 \times l_2$. 
 Then $T_{L/F}(F)/R \simeq  T_{\ell/\kappa}(\kappa)/R.$ 
 In particular, if $T_{l/\kappa}(\kappa)/R = \{1\}$, then $T_{L/F}(F)/R = \{1\}.$

\end{lemma}

\begin{proof}
Let $\nu$ be the discrete valuation on $F$. Let $A= \{x\in F| \nu(x)\geq 0\}$.
Then $A$ is a complete discrete valuation ring with fraction field $F$ and residue field $\kappa$. In particular, $A$ is a complete regular local ring. Since $L_i/F$ are unramified extensions, the multinorm torus $T_{L/F}$ is defined over $A$. Now the result follows from \cite[Proposition 2.2]{gille2004specialization}.

\end{proof}

\begin{cor}
\label{residue_to_cdvf}
Let $K$ be a complete discretely valued field with residue field $\kappa$. Let $K_1,K_2$ be two degree $p$ \'etale algebras over $K$.  Assume that 
$\text{char}(\kappa)\neq p$. If for every choice of two degree $p$ Galois extensions $\kappa_1,\kappa_2$ of $\kappa$, 
$T_{\kappa_1\times \kappa_2/\kappa}(\kappa)/R=\{1\}$, then 
$T_{K_1\times K_2/K}(K)/R=\{1\}$.
\end{cor}

\begin{proof}
If either of the \'etale $K$-algebras, say $K_1$, is not a field extension over $K$, then $K_1$ is a product of $p$ copies of $K$. Then $T_{K_1\times K_2/K}(K)/R=\{1\}$. Thus we may assume that both $K_1$ and $K_2$ are degree $p$ field extensions. If exactly one of these extensions is unramified then we have the result by Lemma \ref{1unramified_1ramified_ext_tori_mod_R_trivial}.  If both of these extensions are ramified, we have the result by Lemma \ref{both_ramified_ext_tori_mod_R_trivial}. If both the extensions are unramified, the result follows from Lemma \ref{both_unramified_ext_tori_mod_R_trivial}.

\end{proof}

\begin{prop}
\label{branch}

Let $K$ be a complete discretely valued field with
residue field $\kappa $ and
$F$ be
the function field of a curve over $K.$
Let $\XX$ be a regular proper model of $F$ with
the reduced special fibre $X$
a union of regular curves with normal crossings.
Let $L_1/F$ and $L_2/F$ be two cyclic extensions of degree $p$. Let $L=L_1\times L_2$.  Assume that $p \neq \text{char}(\kappa)$.
Let $P \in X$ be a closed point and
$U$ an irreducible open subset of
$X$ with $P$ in the closure of $U.$
Let $L = L_1 \times L_2$.
If for any two Galois extensions
 $l_1,l_2$ of $\kappa(P)$ of degree $p,$
 $$\displaystyle T_{l_1\times l_2/\kappa(P)}(\kappa(P))/R =
\{1\}$$ then
 $$ \displaystyle T_{L/F}(F_{P,U})/R = \{1\}. $$

\end{prop}

\begin{proof}Let $\kappa(U)$ be the function field of $U$.
Since $X$ is a union of regular curves, the closed point $P$ gives a discrete valuation on $\kappa(U)$.
Let $\kappa(U)_P$ denote the completion of $\kappa(U)$ at $P$.
Then the branch field $F_{P,U}$ is a complete discretely valued field with residue field $\kappa(U)_P$, and $\kappa(P)$ is the residue field of $\kappa(U)_P$. Thus applying Corollary \ref{residue_to_cdvf} twice, we obtain the result.

\end{proof}

\begin{cond}
\label{A}
We say that a field $\kappa$ satisfies Condition \ref{A}  with respect to a prime $p$ if for all finite extensions $\kappa'/\kappa$ and any two degree $p$ cyclic extensions $l_1,l_2$ of $\kappa'$, $T_{l_1\times l_2/\kappa'}(\kappa')/R = \{1\}$.\\
\end{cond}

\begin{remark}
\label{cd_leq_one}
If $\kappa$ is a field of cohomological dimension $\leq 1$ (for example, if $\kappa$ is an algebraically closed field 
or a finite field) and $\text{char}(\kappa)\neq p$ 
then $\kappa$ satisfies Condition \ref{A} with respect to $p$ (see \cite[Corollaire 6, p-202]{colliot1977r}). 
\end{remark}

\begin{cor}
\label{Sha_P_trivial_multinorm_two_degree_p_ext}

Let $K$ be a complete discretely valued field
with residue field $\kappa.$
Let $F$ be the function field of a curve over $K$.
Let $L_1,L_2$ be two degree $p$
cyclic extensions of $F$.
Assume that $p \neq \text{char}(\kappa).$
Let $L = L_1 \times L_2$.
If $\kappa$ satisfies Condition \ref{A} with respect to $p$, then for any regular proper model
$\XX$ such that the reduced special fibre $X$ is a union of
regular curves with normal crossings, and any choice of $\PP$,
$\Sh_{\PP}(F,T_{L/F}) = 0$
and $\Sh_{X}(F,T_{L/F})=0$.

\end{cor}

\begin{proof}

Let $\XX$ and $X$ as in the statement of the corollary. 
Then by Proposition \ref{branch} and our assumption on $\kappa$
we get that for any branch $(P,\eta)$,
$T_{L/F}(F_{P,\eta})/R =\{1\}$.
Hence, for any choice of $\PP$,
$\Sh _{\PP}(F, T_{L/F}) = 0$
by Theorem \ref{Sha_p_in_terms_of_mod_R}.
Thus $\Sh_{X}(F, T_{L/F}) = 0$
by Theorem \ref{HHK_theorem1}. 
\end{proof} 
As a consequence,
we have the following theorem$\colon$

\begin{theorem}
\label{Sha_dvr_trivial_multinorm_two_degree_p_ext}
Let $K$ be a complete discretely valued field
with residue field $\kappa.$
Let $F$ be the function field of a curve over $K$.
Let $L_1,L_2$ be two degree $p$
cyclic extensions of $F$.
Assume that $p \neq \text{char}(\kappa).$
Let $L= L_1 \times L_2.$
If $\kappa$ satisfies Condition \ref{A} with respect to $p$, then $\Sh(F, T_{L/F}) = 0.$

\end{theorem}

\begin{proof}
 The result follows from Corollary \ref{Sha_P_trivial_multinorm_two_degree_p_ext} and Theorem \ref{union_of_Sha_X_equals_Sha_dvr_multinorm}.
\end{proof}

We provide a basic lemma here which will be used in the next result. 

\begin{lemma}
\label{existence_of_M}
Let $F$ be a field and $\overline{F}$ be an algebraic closure of $F$. Let $F \subset L_1,L_2 \subseteq \overline{F}$ be fields such that both $L_1/F$ and $L_2/F$ are degree $p$ separable extensions. Then there exists a field extension $M/F$ such that $M\subseteq \overline{F}$ with $[M:F]$ coprime to $p$, and $L_1M/M$ and $L_2M/M$ are both degree $p$ Galois extensions.
\end{lemma}

\begin{proof}
Let $N_1\subseteq \overline{F}$ be the Galois closure of $L_1/F$. Then $G_1:=\mathsf{Gal}(N_1/F)\subseteq S_p$ and there exists an order $p$ subgroup $H_1$ of $G_1$. Let $M_1$ be the fixed field of $H_1$. Then $p\nmid [M_1:F]$ and $[L_1M_1:M_1]\leq [L_1:F]=p$, which implies that $L_1M_1=N_1$, which is a degree $p$ Galois extension over $M_1$.

Note that $L_2M_1/M_1$ is also a degree $p$ extension. If this extension is Galois, we are done. If not, then arguing as in the above paragraph, we get that there exists a field  $M$, $M_1 \subseteq M \subseteq \overline{F}$, such that the extension $M/M_1$ is of degree coprime to $p$ and $L_2M/M$ is a Galois extension of degree $p$. Since $M/M_1$ is of degree coprime to $p$, $L_1M/M$ is also a Galois extension of degree $p$. Thus the extension $M/F$ has the required properties.

\end{proof}

\begin{cor}
\label{any_two_deg_p}

Let $K$ be a complete discretely valued field
with residue field $\kappa.$
Let $F$ be the function field of a curve over $K$.
Let $L_1,L_2$ be two degree $p$
separable extensions of $F$.
Assume that $p \neq \text{char}(\kappa).$
Let $L= L_1 \times L_2.$
If $\kappa$ satisfies Condition \ref{A} with respect to $p$, then $\Sh(F, T_{L/F}) = 0.$

\end{cor}

\begin{proof} We may assume that both $L_1$ and $L_2$ are contained in a fixed algebraic closure $\overline{F}$ of $F$. By Lemma \ref{existence_of_M}, there exists an extension $M/F$ of degree $n$, $M \subseteq \overline{F}$, such that $L_1M/M$ and $L_2M/M$ are both degree $p$ Galois extensions and $n$ is coprime to $p$. 
Let $\lambda \in F^{\times} \subseteq M^{\times}$. 
If $\lambda $ is a norm from $L\otimes_F F_{\nu}$ to $F_{\nu}$ for all discrete valuations $\nu$ on $F$, then $\lambda $ is a norm from $(L_1M \times L_2M) \otimes_F F_{\nu}$ to $M\otimes_F F_{\nu}$ for all $\nu$. 
By Theorem \ref{Sha_dvr_trivial_multinorm_two_degree_p_ext}, the local-global principle holds for the multinorm torus $T_{L_1M\times L_2M/M}$. Thus $\lambda$ is a norm from $L_1M\times L_2M$ to $M$. Then $\lambda^n$ is a norm from $L_1M\times L_2M$ to $F$. This implies that $\lambda^n$ is a norm from $L_1\times L_2$ to $F$. Now, clearly $\lambda^p $ is a norm from $L_1\times L_2$ to $F$ (for example, $N_{L_1/F}(\lambda) N_{L_2/F}(1)=\lambda^p$). Since $n$ is coprime to $p$ and $\lambda^n, \lambda^p$ both are norms from $L_1\times L_2$ to $F$, we conclude that $\lambda$ is a norm from $L_1\times L_2$ to $F$.
\end{proof}

\begin{remark}

We note that we do not need any assumptions on the
graph associated to regular proper models of the semi-global field $F$ in the statement of the above corollary.

\end{remark}

\begin{cor}
\label{lgp_two_deg_p_cyclic_general}

Let $\kappa_0$ be a field of cohomological dimension 
$\leq 1$. 
Let $\kappa$ be an iterated Laurent series in $r$ variables over $\kappa_0$ where $r \geq 0$.
Let $K$ be a complete discretely valued field
with residue field $\kappa.$
Let $F$ be the function field of a curve over $K$.
Let $L_1,L_2$ be two degree $p$ separable extensions of $F$.
Assume that $p \neq \text{char}(\kappa_0).$
Let $L= L_1 \times L_2.$ Then $\Sh(F, T_{L/F}) = 0.$

\end{cor}

\begin{proof}
By Corollary \ref{any_two_deg_p}, it is enough to show that $\kappa$ 
satisfies Condition \ref{A} with respect to $p$. 
We prove this by induction on $r$.
If $\kappa =\kappa_0$ then $\kappa$ satisfies Condition \ref{A} with respect to $p$ by Remark \ref{cd_leq_one}.

Suppose the result is known for $\kappa_1$ which is the Laurent series over $\kappa_0$ in $r=k$ variables. Let $\kappa$ be the Laurent series over $\kappa_1$ in one variable. Then $\kappa$ is a complete discretely valued field with residue field $\kappa_1$. Applying Corollary \ref{residue_to_cdvf},  we conclude that $\kappa$
also satisfies the Condition \ref{A} with respect to $p$.
\end{proof}

Theorem \ref{Sha_dvr_trivial_multinorm_two_degree_p_ext} can be generalized for the case of two abelian extensions of degree $n$ where $n$ is a squarefree integer not divisible by $\text{char}(\kappa)$. \\

\begin{theorem}
\label{lgp_two_extenions_of_same_degree_squarefree_degree}

Let $K$ be a complete discretely valued field
with residue field $\kappa.$
Let $F$ be the function field of a curve over $K$.
Let $L_1$ and $L_2$ be two abelian extensions of $F$ of degree $n$,
where $n$ is a squarefree integer not divisible by $\text{char}(\kappa)$. Let $L= L_1 \times L_2.$
If $\kappa$ satisfies Condition \ref{A} with respect to all $p$ dividing $n$, then the local-global principle holds for the multinorm torus $T_{L/F}$ with respect to discrete valuations.

\end{theorem}

\begin{proof}
Let $n = p_1p_2 \cdots p_m$ where $p_i$ are distinct primes.
Let $G_1$ (resp. $G_2$) denote the Galois group of $L_1$(resp. $L_2$).
Let $H_i$ (resp. $H'_i$), $1 \leq i \leq m$, denote the Sylow $p_i$ subgroup of $G_1$(resp. $G_2$).
Let $M_i = L_1^{H_i}$ and $N_i = L_2^{H'_i}$ for $i$, $1\leq i \leq m$.\\ 

We prove the result by induction on $m$, the number of distinct prime factors of $n$.
For $m=1$, the result is true by Theorem \ref{Sha_dvr_trivial_multinorm_two_degree_p_ext}.
Suppose we know the result for $m=r$, $r \geq 1$.
We now argue that the result holds for $m=r+1$ case.\\ 

Let $\lambda \in F^{\times}$ such that $\lambda \in N_{(L_1 \otimes_F F_{\nu}) \times(L_2 \otimes_F F_{\nu})/F_{\nu} }((L_1 \otimes_F F_{\nu}) \times(L_2 \otimes_F F_{\nu}))$ for all $\nu \in \Omega_F$.
Then $\lambda \in N_{(M_i \otimes_F F_{\nu}) \times(N_i \otimes_F F_{\nu})/F_{\nu} }((M_i \otimes_F F_{\nu}) \times(N_i \otimes_F F_{\nu}))$ for all $\nu \in \Omega_F$ and for all $i$, $1 \leq i \leq r+1$. By induction hypothesis, this implies that  $\lambda \in N_{M_i \times N_i/F}(M_i \times N_i)$ for  all $i$, $1 \leq i \leq r+1$. Thus $\lambda ^{p_i} \in N_{L_1 \times L_2 /F}(L_1 \times L_2)$ for all $i$, $1 \leq i \leq r+1$.
In particular, $\lambda^{p_1}, \lambda^{p_{r+1}} \in N_{L_1 \times L_2 /F}(L_1 \times L_2).$
Since $p_1$ and $p_{r+1}$ are distinct primes, we conclude that $\lambda \in N_{L_1 \times L_2 /F}(L_1 \times L_2)$.
\end{proof}

\begin{remark} The above result no longer holds if we consider $n$ which is not squarefree. For example, when $n$ is a square, we have counterexamples (see \cite[Corollary 7.12]{mishra2021local}).\end{remark}

In other direction, we have following result$\colon$

\begin{theorem}
\label{lgp_multinorm_p_general_case}

Let $K$ be a complete discretely valued field
with residue field $\kappa.$
Let $F$ be the function field of a curve over $K$.
Let $L_1$ be a degree $p$ separable extension of $F$.
Let $L_2$ be a Galois extension of $F$ of degree $n$ with Galois group $G$.
Assume that $p^2$ does not divide $n$.
If $p$ divides $n$ then also assume that $G$ has an index $p$ subgroup $H$.
Assume that  $\text{char}(\kappa)$ does not divide $pn$.
Let $L= L_1 \times L_2.$
If $\kappa$ satisfies Condition \ref{A} with respect to $p$, then the local-global principle holds for $T_{L/F}$ with respect to discrete valuations.

\end{theorem}

\begin{proof}

If $p$ does not divide $n$ then the local-global principle holds by Observation \ref{coprime_degrees}.
So we can assume that $n = pr$ for some positive integer $r$ with
$r$ coprime to $p$.\\

Let $L'_{2}$ be the fixed field of $H$ in $L_2$.
Assume that $\alpha \in F^{\times}$ is a norm from
$L \otimes_F F_{\nu}$ to $F_{\nu}$ for all $\nu \in \Omega_F$.
Then $\alpha$ is also a norm from
$(L_1 \times L'_{2}) \otimes_F F_{\nu} $ to $F_{\nu}$
for all $\nu \in \Omega_F$.
By Corollary \ref{any_two_deg_p},
$\alpha$ is a norm from $L_1 \times L'_2$ to $F$.
Then $\alpha^{r}$ is a norm from $L$ to $F$.\\

Now let $E$ be an order $p$ subgroup of $G$.
Let $L^{\prime \prime}_{2}$ be the fixed field of $E$ in $L_2$.
Since $\alpha = N_{F/F}(\alpha) N_{L^{\prime \prime}_{2}/F}(1)$, $\alpha$ is a norm from
$F \times L^{\prime \prime}_{2}$ to $F$.
Hence $\alpha^{p}$ is a norm from $L$ to $F$.
Since $r$ and $p$ are coprime and $\alpha^{r}$ is a norm
from $L$ to $F$, we conclude that $\alpha$ is a norm from $L$ to $F$.

\end{proof}

For $p=2$, the following stronger statement holds$\colon$

\begin{theorem}
\label{lgp_two_quadratic}

Let $K$ be a complete discretely  valued field with residue field $\kappa$.
Let $F$ be function field of a curve over $K$.
Let $L_1$ and $L_2$ be two quadratic extensions of $F$.
Assume that $2 \neq \text{char}(\kappa)$.
Let $\displaystyle L = L_1 \times L_2.$
Then the local-global principle holds for $T_{L/F}$ with respect to discrete valuations.

\end{theorem}

\begin{proof}

Since norm is multiplicative, the multinorm torus associated to two quadratic extensions $F(\sqrt{a})/F$ and $F(\sqrt{b})/F$  is given by the equation$:$ $X^2_1 -a X^2_2 = X^2_3 -bX^2_4$ or equivalently, $X^2_1 -a X^2_2 - X^2_3 +bX^2_4 = 0$. Since  $X^2_1 -a X^2_2 - X^2_3 +bX^2_4$ is a smooth quadric with rational points, e.g., $(1,0,1,0)$, the multinorm torus is rational.
Hence by \cite[Corollary 6.5]{harbater2015local}, $\Sh_{X}(F,T_{L/F})=0$. Now the result follows from Theorem \ref{union_of_Sha_X_equals_Sha_dvr_multinorm}.

\end{proof}

Before our next result, we recall a basic fact from theory of finite groups.

\begin{lemma}
\label{2m}
Let $m$ be an odd number and let $G$ be a finite group of order $2m$. Then there exists a normal subgroup of $G$ of index $2$.
\end{lemma}
\begin{proof}
Since $|G|=2m$, by Cauchy's theorem, there exists an element in $G$, say $a$, of order $2$. 
Let $\mathsf{Sym}(G)$ denote the group of permutations of the underlying set of $G$. 
Let $\phi: G\rightarrow \mathsf{Sym}(G)$ be the group homomorphism defined by sending $g\in G$ to $\phi_g$, where $\phi_g(x)=gx$ for all $x\in G$. We note that the permutations $\phi_g$ do not fix any element $x$ in $G$. Also, since $a$ is of order $2$ in $G$, $\phi_a$ is of order $2$ in $\mathsf{Sym}(G)$. Thus $\phi_a$ is a disjoint product of $m$ many transpositions. 

Let $\mathsf{sgn}: \mathsf{Sym}(G) \rightarrow \{\pm1\}$ be the signature homomorphism. Consider the group homomorphism $\mathsf{sgn} \circ \phi : G \rightarrow \{\pm 1\}$. Then $\mathsf{sgn} \circ \phi $ is surjective as $(\mathsf{sgn} \circ \phi)(a) =(-1)^m=-1.$ Thus $\mathsf{ker}(\mathsf{sgn} \circ \phi )$ is a normal subgroup of $G$ of index $2$.
\end{proof}

\begin{theorem}

 Let $K$ be a complete discretely  valued field with residue field $\kappa$.
Let $F$ be function field of a curve over $K$. Let $L_1$ be a quadratic extension of $F$.
Let $L_2$ be a Galois extension of $F$ of degree $n$ with Galois group $G$.
Assume that $4$ does not divide $n$.
Assume that $\text{char}(\kappa)$ does not divide $2n$.
Let $\displaystyle L = L_1 \times L_2.$
Then the local-global principle holds for $T_{L/F}$ with respect to discrete valuations.
\end{theorem}
 \begin{proof} 
 By Lemma \ref{2m}, $G$ has an index $2$ subgroup. Arguing as in case of proof of Theorem \ref{lgp_multinorm_p_general_case} and using Theorem \ref{lgp_two_quadratic}, one gets the result.
 \end{proof}

\section{Local-Global Principles for Product of norms from more than two extensions}
\label{LGP_2}

In this section, we prove a local global principle for some classes of multinorm tori over semiglobal fields such that the graph associated to their regular proper model is a tree.
Let $K$ be a complete discretely valued field
with residue field $\kappa$ algebraically
closed. Let $F$ be the function field of a curve over $K$.
Let $T_{L/F}$ be the multinorm torus associated to
cyclic extensions $L_i/F$, $1 \leq i \leq m$
with degrees $[L_i \colon F]=p$ where $p$ is a prime number.
We study $T_{L/F}(F_{P,\eta})/R$ for
branches $(P,\eta)$.

\begin{notation}

For a discrete valuation $\nu$ on $F$,
let $F_{\nu,h} \subseteq F_{\nu}$ be the henselization of $F$ at $\nu$. 
\end{notation}

We start with some results concerning
tori defined over semi-global fields.\\

\begin{lemma}
\label{surjectivity_for_Rtrivial}
Let $T$ be a torus defined over
a semi-global field $F$. 
Let $\XX$ be a regular proper model of $F$
such that the reduced special fibre $X$ is a union of regular curves with normal crossings. 
Let $V \subsetneq X$ be a non-empty open subset of $X$ contained in an irreducible component $X_{\eta}$ of $X$. 
Suppose there exists a closed point $P$ of $X_{\eta}$ such that $V\cup \{P\} \subsetneq X_{\eta}$ and does not intersect any other irreducible component of $X$. 
Then the natural map $RT(F_V) \times  RT(F_P) \rightarrow RT(F_{P,\eta})$ is surjective.
\end{lemma}

\begin{proof}
Let $1 \rightarrow S \rightarrow Q \rightarrow T \rightarrow 1$ be a flasque resolution of the torus $T$ over $F$, where $Q$ is a quasitrivial torus defined over $F$ and $S$ is a flasque torus defined over $F$ (see \cite{colliot1977r} and \cite{colliot1987principal}). 
Let $V'=V \cup \{P\}.$
Since $Q$ is quasi-trivial, $H^1(F_{V'},Q)=0$. We have $F_{V'} \subset F_{V}, F_{P} \subset F_{P,\eta}$ and we get that $F_{V},F_{P}, F_{P,\eta}$ form an inverse system of fields over $F_{V'}$, as  in \cite{harbater2015local}.
Hence by \cite[Theorem 2.4]{harbater2015local} and \cite[Proposition 3.9]{harbater2015refinements}, it follows that 
$Q(F_V)\times Q(F_P) \rightarrow Q(F_{P,\eta})$ is surjective. 
Since the image of  the map $Q(L) \rightarrow T(L)$ is $RT(L)$ for every field extension $L/F$ (see \cite[Theorem 2]{colliot1977r} and its proof), the map $RT(F_V) \times RT(F_P) \rightarrow RT(F_{P,\eta})$ is surjective.

\end{proof}

\begin{lemma}
\label{F_U_to_F_eta_mod_R}
Let $F$, $T$, $\XX$, $X$ as in Lemma \ref{surjectivity_for_Rtrivial}. Suppose that$\colon$

\begin{itemize}
\item $a)$ the natural map $T(F_{\eta,h}) \to T(F_\eta)/R$ is
 surjective for all codimension zero points $\eta$ of $X$, and,

\item $b)$ the natural map $T(F_P) \rightarrow T(F_{P,\eta})/R$
is surjective for all codimension zero point $\eta$ of $X$ and for all $P$ such that $P$ is contained in $\eta$.
\end{itemize}
Let $U$ be a non-empty proper open subset of $X_{\eta}$ such that $U$ does not meet another irreducible components of $X$.
Then the map $T(F_U) \rightarrow T(F_{\eta})/R$
is surjective.

\end{lemma}

\begin{proof}

 Let $U$ be a non-empty proper open subset of $X_{\eta}$ such that $U$ does not meet another irreducible components of $X$ and  $\mu \in T(F_{\eta}).$
Then by the assumption $a)$, there exists $\mu_h \in T(F_{\eta,h})$
which maps to $\mu$ in $T(F_\eta)/R$.
Thus, replacing $\mu$ by $\mu_h$, we may assume
that $\mu \in T(F_{\eta,h})$.
Since, by \cite[Lemma 3.2.1.]{harbater2014local},
the field $F_{\eta,h}$ is the filtered direct limit of the fields $F_V$,
where $V$ ranges over the nonempty open proper subsets of $X_\eta$, there exists a non-empty proper open subset $V$ of $X_{\eta}$ such that $V$ does not meet another irreducible components of $X$ and $\mu \in T(F_V)$. By taking intersection with $U$, if needed,
we may assume that $V \subseteq U.$

If $V = U$, there
is nothing to prove.
Suppose that $V\neq U$.
Let $P \in U \setminus V$.
By assumption, the map
$T(F_P) \rightarrow T(F_{P,\eta})/R$
is surjective.
Let $\mu_P$ be an element in
$T(F_P)$ mapping to the class of
$\mu$ in $T(F_{P,\eta})/R$.
Then $\mu = \mu_P \alpha_{P,\eta}$
for some $\alpha_{P,\eta} \in RT(F_{P,\eta})$. By Lemma \ref{surjectivity_for_Rtrivial}, there exists $\alpha_V \in RT(F_V)$ and
$\alpha_P \in RT(F_P)$ with
$\alpha_{P,\eta} = \alpha_P \alpha_V$
in $RT(F_{P,\eta}).$
Then $\mu = \mu_P \alpha_P \alpha_V$.
Hence $\mu \alpha^{-1}_V = \mu_P \alpha_P
\in T(F_V) \cap T(F_P)$.
By \cite[Lemma 2.12.]{harbater2015refinements}
and \cite[Proposition 3.9.]{harbater2015refinements},
$\mu \alpha^{-1}_V \in T(F_{V \cup \{P\}})$.
Also, $\mu \alpha^{-1}_V$ maps to
the equivalence class of $\mu$
in $T(F_{\eta})/R$ since
$\alpha_V \in RT(F_V) \subseteq RT(F_{\eta}).$
Since $U \setminus V$ is a finite set,
doing this process finitely many times, we get the
result.

\end{proof}

\begin{lemma}
\label{branch_to_eta_implies_branch_to_U}
Let $F$, $T$, $\XX$, $X$ as in Lemma \ref{surjectivity_for_Rtrivial}. Suppose that $\colon$ \\
\begin{itemize}
\item the natural map $T(F_{\eta,h}) \to T(F_\eta)/R$
is surjective for all codimension zero points $\eta$ of $X$, and,

\item for every closed point $P$ of $X_\eta$,
the natural maps $T(F_{\eta}) \rightarrow T(F_{P,\eta})/R$ and
$T(F_P) \rightarrow T(F_{P,\eta})/R$ are surjective.
\end{itemize}

Then for any proper open subset $U$ of $X_\eta$,
the natural map $$T(F_{U}) \rightarrow
T(F_{P,\eta})/R$$
is surjective.

\end{lemma}

\begin{proof}

The proof follows immediately from
Lemma \ref{F_U_to_F_eta_mod_R} since
the natural map $T(F_{U}) \rightarrow
T(F_{P,\eta})/R$ factors through
$T(F_{\eta})/R$.

\end{proof}

\begin{prop}
\label{eta_P_implies_ShaX_zero_multinorm}
Let $F$, $T$, $\XX$, $X$ as in Lemma \ref{surjectivity_for_Rtrivial}. Assume that$\colon$\\

$\bullet$ the graph associated to $\XX$
is a tree,\\

$\bullet$ the natural map $T(F_{\eta,h}) \to T(F_\eta)/R$ is surjective for all codimension zero points $\eta$ of $X$, and,\\

$\bullet$ the natural map $T(F_{\eta})
\rightarrow T(F_{P,\eta})/R$ is surjective 
for all possible codimension zero points $\eta$ of $X$ and branches $(P,\eta)$
and, all choices of $\PP$, and\\

$\bullet$ the natural map $T(F_{P})
\rightarrow T(F_{P,\eta})/R$ is surjective
for all possible $P$ and branches $(P,\eta)$
and all choices of $\PP$. \\

Then $\Sh_X(F,T)=0$.

\end{prop}

\begin{proof}

By \ref{HHK_theorem1}, we just need to show that
for every choice of $\PP$,
$\Sh_{\PP}(F,T)$ is trivial.
This immediately follows from
Lemma \ref{branch_to_eta_implies_branch_to_U} and
Corollary \ref{surjective_from_P_and_U_implies_ShaP_trivial}.

\end{proof}

\begin{lemma}
\label{surjective_for_general_extensions}

Let $F$ be a field and $L_i/F$, $1 \leq i \leq m$, be finite separable extensions. 
Let $[L_i \colon F] = n_i$.
Let $\displaystyle L=\prod^{m}_{i=1} L_i$ and let $T_{L/F}$ be the associated
multinorm torus.
Let
$\displaystyle n\colonequals
 \mathsf{lcm} (n_i \mid 1 \leq i \leq m ).$
Let $M/F$ be a field extension.
Assume that:\\

$\bullet$ for all
$(\mu_1, \dots, \mu_m) \in T_{L/F}(M)$,
there exists $(\mu'_1, \dots, \mu'_m) \in T_{L/F}(F)$
such that $N_{L_i\otimes_F M/M}(\mu_i\mu^{'-1}_i)
\in M^{\times n}$ for all $i$, $1 \leq i \leq m$.\\

$\bullet$ The natural maps
$$ T_{L_i/F}(F) \rightarrow T_{L_i/F}(M)/R $$
are surjective for $1 \leq i \leq m$. \\

Then the natural  map
$T_{L/F}(F) \rightarrow T_{L/F}(M)/R$
is surjective.

\end{lemma}

\begin{proof}

Let $(\mu_1, \dots, \mu_m) \in T_{L/F}(M).$
Then by assumption, there exists
$(\mu'_1, \dots, \mu'_m) \in T_{L/F}(F)$
with
$N_{L_i \otimes_F M / M}(\mu_i \mu^{'-1}_i)
= \alpha^{n}_{i}$ for some $\alpha_i \in M$
for $1 \leq i \leq m.$

Let us consider $x_i = \mu_i \mu^{'-1}_i.$
Then it is enough to show that the class
of $(x_1, \dots, x_m)$ is in the image of
the map $T_{L/F}(F) \rightarrow T_{L/F}(M)/R.$ \\

Now we consider
$$\displaystyle (\beta_1, \dots, \beta_{m-1}, \beta_m)
=(\alpha^{n/n_1}_1, \alpha^{n/n_2}_2, \dots,
\alpha^{n/n_{m-1}}_{m-1},
\prod_{i=1}^{m-1}\alpha_{i}^{-n/n_{m}})$$
in $T_{L/F}(M).$
Then by Lemma \ref{R_trivial_elements},
$\displaystyle (\beta_1, \dots, \beta_m) \in RT_{L/F}(M).$
Thus, it is enough to show that the class of
$\displaystyle (x_1 \beta^{-1}_{1}, \dots, x_m \beta^{-1}_{m})$
lies in the image of the map
$\displaystyle T_{L/F}(F) \rightarrow T_{L/F}(M)/R.$
For $1 \leq i \leq m,$
$\displaystyle x_i \beta^{-1}_{i} \in T_{L_i/F}(M).$
Now, by assumption, there exists
$\displaystyle \gamma_i \in T_{L_i/F}(F)$ mapping to
the class of $x_i \beta^{-1}_{i} \in T_{L_i/F}(M)/R$
for $1 \leq i \leq m.$
Then $(\gamma_1, \dots, \gamma_m)$ maps to
the class of
$\displaystyle (x_1 \beta^{-1}_{1}, \dots, x_m \beta^{-1}_{m})$
in $T_{L/F}(M)/R.$ Hence we are done.

\end{proof}

\begin{lemma}
\label{Henselian_to_completion_singleext_surjection}

Let $F$ be a Henselian
discretely valued field
with residue field $\kappa$ and
with completion $\hat{F}$.
Assume that $F$ has a
primitive $n^2$-th root of unity $\rho_{n^2}$.
Let $L/F$ be a finite separable
field extension of degree $n$ which is
not divisible by $\text{char}(\kappa)$.
Let $T_{L/F}$ be
the associated norm one torus.
Then the natural map
$$T_{L/F}(F) \to T_{L/F}(\hat{F})/R$$
is surjective.

\end{lemma}

\begin{proof}
Since $F$ is Henselian,
$L$ is also Henselian and
$L\otimes_F \hat{F}$ is a field and
completion of $L$.
Let $u \in T_{L/F}(\hat{F}).$
Then $u$ is a unit of norm $1$ on $L\otimes_F \hat{F}$, under the identification of $T_{L/F}(\hat{F})$ with the norm $1$ elements.
Since $n$ is not divisible by $\text{char}(\kappa)$, if $v \in L$ is sufficiently close to $u$ in $L\otimes_F \hat{F}$ then $u = v a^{n^2}$ for some $a \in L\otimes_F \hat{F}$.
Since $N_{L/F}(v) [N_{L\otimes_F \hat{F}/\hat{F}}(a)]^{n^2}
= N_{L\otimes_F \hat{F}/\hat{F}}(u) =  1$,
$N_{L/F}(v) \in \hat{F}^{\times n^2}$.
Since $F$ is Henselian, there exists
$b \in F$ such that $N_{L/F}(v) = b^{n^2}$.
Then $N_{L/F}(vb^{-n})=1.$
We have $ u (vb^{-n})^{-1} = (a^nb)^{n}
\in T_{L/F}(\hat{F})$.
Since  $N_{L\otimes_F \hat{F}/\hat{F}}(a^nb)^n = 1$,
we have $N_{L\otimes_F \hat{F}/\hat{F}}(a^nb)  = \rho^{ni}_{n^2}$
for some $i$, $0 \leq i \leq n-1$.
Then $N_{L\otimes_F \hat{F}/\hat{F}}(a^nb\rho^{-i}_{n^2}) = 1$.
Let us consider $w = vb^{-n} \rho^{-ni}_{n^2}$. Then 
$N_{L/F}(w)=1$ and $w \in T_{L/F}(F)$ and
$uw^{-1} = (a^n b \rho^{i}_{n^2})^n.$
Since $(a^n b \rho^{i}_{n^2})^n \in RT_{L/F}(\hat{F})$
\cite[Proposition 15]{colliot1977r},
we conclude that $uw^{-1} \in RT_{L/F}(\hat{F})$
and hence we are done.

\end{proof}

\begin{lemma}
\label{F_to_F_Henselian}
Let $F$ be a Henselian discretely valued field
and $\hat{F}$ be its completion. Assume that the valuation ring of $F$ is excellent. 
Let $\kappa$ be the residue field of $F$.
Let $L_1, L_2, \dots, L_m$ be finite separable
extensions of $F$ with $[L_i \colon F] = n_i$.
Assume that $F$ has all primitive $n^2_i$-\rm{th}
roots of unity for all $i$, $1 \leq i \leq m$.
Assume that $\text{char}(\kappa)$ does not divide $n_i$
for $1 \leq i \leq m$.
Let $\displaystyle L = \prod_{i=1}^{m} L_i$.\\
Then $$T_{L/F}(F) \to T_{L/F}(\hat{F})/R$$
is surjective.
\end{lemma}

\begin{proof}
Let $\mu \in T_{L/F}(\hat{F})/R$.
Then $\mu = (\mu_1, \dots ,\mu_m)$ with
$\mu_i \in L_i \otimes_F \hat{F}$
and $\displaystyle \prod_{i=1}^{m}
N_{L_i \otimes_F \hat{F}/\hat{F}}(\mu_i) = 1$.
Let
$\displaystyle n =
 \mathsf{lcm} (n_i \mid 1 \leq i \leq m ).$\\

We first show that for all
$(\mu_1, \dots, \mu_m) \in T_{L/F}(\hat{F})$,
there exists $(\mu'_1, \dots, \mu'_m) \in T_{L/F}(F)$
such that $N_{L_i\otimes_F \hat{F}/\hat{F}}
(\mu_i\mu^{'-1}_i) \in \hat{F}^{\times n}.$
Since $L_i \otimes_F \hat{F}$ is the completion
of $L_i$ for $1 \leq i \leq m-1$,
there exists $\mu_i' \in L_i$ with
$\mu_i = \mu_i' \theta_i^n$ for
some $\theta_i \in (L_i \otimes_F \hat{F})^{\times n}$.
Let $\displaystyle \lambda =
\prod_{i=1}^{m-1} N_{L_i /F}(\mu'_i)  \in F$.
Since $\displaystyle \lambda (\prod_{i=1}^{m-1}
N_{L_i \otimes_F \hat{F}/\hat{F}}(\theta_i)^n) =
\prod_{i=1}^{m-1} N_{L_i \otimes_F \hat{F}/\hat{F}}(\mu_i) =
N_{L_m \otimes_F \hat{F}/\hat{F}}(\mu_m)^{-1}$,
$\lambda$ is a norm from $L_m\otimes_F \hat{F}/\hat{F}$.
Since $F$ is Henselian and $\hat{F}$ is the completion of $F$,
 $\lambda $ is a norm from $L_m/F$ by
 \cite[Thm \Romannum{1}.\Romannum{1}$0$]{artin1969algebraic}.
Let $\mu'_m \in L_m$ with
$N_{L_m /\hat{F}}(\mu'_m) = \lambda^{-1}.$
 Then $\mu' = (\mu_1', \dots,\mu_m') \in T_{L/F}(F)$ and
 $N_{L_i\otimes_F \hat{F}/\hat{F}}(\mu_i\mu_i^{'-1})
 \in \hat{F}^{\times n}$ for $1 \leq i \leq m$.\\

 Now, by Lemma \ref{Henselian_to_completion_singleext_surjection},
 we get that the maps $T_{L_i/F}(F) \to T_{L_i/F}(\hat{F})/R$
 are surjective for all $i$, $1 \leq i \leq m$.
 Hence, by Lemma \ref{surjective_for_general_extensions},
 $T_{L/F}(F) \to T_{L/F}(\hat{F})/R$ is surjective.

 \end{proof}

\begin{remark}
If all the extensions $L_i/F$, $1 \leq i \leq m$
are cyclic of degree $n_i$,
then we do not need to assume that
$F$ has all primitive $n^2_i$-\rm{th}
roots of unity for all $i$, $1 \leq i \leq m$.
Since, by Lemma \ref{R_trivial_elements_cyclic_norm_one_tori},
$T_{L_i/F}(\hat{F})/R =\{1\}$.
\end{remark}

\begin{lemma}
\label{approximating_norm_from_eta_to_branch_multinorm}

Let $K$ be a complete discretely valued field
with residue field $\kappa$.
Let $F$ be the function field of a curve over $K$.
Let $\XX$ be a regular proper model of $F$
such that the reduced special fibre $X$ is a union of
regular curves with normal crossings. 
Let $\mathcal{P} \in X$ be a nonempty
finite set of closed points
containing all the intersection points.
Let $\eta$ be the generic point of one
of the components of $X$ and
let $(P,\eta)$ be a branch.
Let $L_{\eta}/F_{\eta}$ be a field extension of degree $n$ with
$n$ not divisible by $\text{char}(\kappa)$.
Assume that $L_{P,\eta} =
L_{\eta}\otimes_F F_{p,\eta}$ is a field.
Then for any $\lambda \in
N_{L_{P,\eta}/F_{P,\eta}}((L_{P,\eta})^{\times})$
and an integer
$N$ not divisible by $\text{char}(\kappa)$,
there exists a
$\lambda' \in N_{L_{\eta}/F_{\eta}}((L_{\eta})^{\times})$
such that
$\lambda \lambda'^{-1}$ is a
$N^{\rm{th}}$ power in $F_{P,\eta}^{\times}.$

\end{lemma}

\begin{proof}

Let $L(\eta)$ and $F(\eta)$ denote the residue field
of $L_{\eta}$ and $F_{\eta}$, respectively.
Similarly, $L(P,\eta)$ and $F(P,\eta)$ denote
the residue field of $L_{P, \eta}$ and $F_{P,\eta}$,
respectively.
Then $F(P,\eta)$ and $L(P,\eta)$ are completions of
$F(\eta)$ and $L(\eta)$ respectively.
Let $L^{nr}_{\eta}$ and $L^{nr}_{P,\eta}$ be the maximal
unramified subextensions of $L_{\eta}/F_{\eta}$ and
$L_{P,\eta}/F_{P,\eta}$ respectively.
Since $[L_{P,\eta} \colon L^{nr}_{P, \eta}]
= [L_{\eta} \colon L^{nr}_{\eta}] = e$ and
$[L_{P,\eta} \colon F_{P,\eta}]
= n = [L_{\eta} \colon F_{\eta}],$
we get that $
[L(\eta) \colon F(\eta)] = [L^{nr}_{\eta} \colon F_{\eta}]
= n/e = [L^{nr}_{P,\eta} \colon F_{P,\eta}] = [L(P,\eta) \colon F(P,\eta)].$\\

Let $\pi \in F_{\eta}$ be a parameter. 
Since $L_{\eta}/L^{nr}_{\eta}$ is totally ramified,
we can write $L_{\eta}$ as
$L_{\eta} = L^{nr}_{\eta}(\sqrt[e]{u \pi})$,
where $u$ is a unit in $L^{nr}_{\eta}$
by \cite[Lemma 2.4]{parimala2018local}.
Then $L_{P,\eta} = L^{nr}_{P,\eta}(\sqrt[e]{u \pi})$.\\

Let $\mu \in L_{P,\eta}^{\times}$
with $\lambda = N_{L_{P,\eta}/F_{P,\eta}}(\mu)$.
Then we can write
$\mu = \theta (\sqrt[e]{u \pi} )^{i},$
where $\theta \in L_{P,\eta}$ is a unit
and $i$ is an integer.
Since $L_{P,\eta}$ and $L^{nr}_{P,\eta}$
have same residue field $L(P,\eta)$,
and since $N$ is not divisible by $\text{char}(\kappa)$,
by Hensel's lemma, there exists
$\theta'$ a unit in $L^{nr}_{P,\eta}$
and $\alpha \in L_{P,\eta}$
with $\theta = \theta'\alpha^{N}$.
Hence without loss of generality
we can assume that
$\theta \in L^{nr}_{P,\eta}.$
We have
$$\begin{array}{rcl}
\displaystyle \lambda =  N_{L_{P, \eta}/F_{P, \eta}}(\mu) & = &
[N_{L^{nr}_{P,\eta}/F_{P,\eta}}(\theta)]^{e}
[N_{L_{P,\eta}/F_{P,\eta}}(\sqrt[e]{u \pi} )]^{i}\\
& = & [N_{L^{nr}_{P,\eta}/F_{P,\eta}}(\theta)]^{e}
[N_{L_{\eta}/F_{\eta}}(\sqrt[e]{u \pi})]^{i}.
\end{array}
$$

Let $\overline{\theta}$ be the image of $\theta$
in the residue field $L(P,\eta)$.
We have $\overline{N_{L^{nr}_{P,\eta}/F_{P,\eta}}(\theta)} =
N_{L(P,\eta)/F(P,\eta)}(\overline{\theta})  \in F(P,\eta).$
Since $L(P,\eta)$ is completion of $L(\eta)$
and by Hensel's lemma,
there exists $\overline{\phi} \in L(\eta)$ such that
$\overline{\theta} \cdot  \overline{\phi}^{-1}$ is a
$N^{\rm{th}}$ power in $L(P,\eta)$.
Let $\phi \in L^{nr}_{\eta}$ such that
$\overline{\phi}$ is the image of $\phi$ in $L(\eta)$.
Then $N_{L(\eta)/F(\eta)}(\overline{\phi}) =
\overline{N_{L^{nr}_{\eta}/F_{\eta}}(\phi)}.$
We get that
$\overline{N_{L^{nr}_{P,\eta}/F_{P,\eta}}(\theta \phi^{-1})}$
is a $N^{\rm{th}}$ power in $F(P,\eta)$.
Hence, by Hensel's lemma, $$N_{L^{nr}_{P,\eta}/F_{P,\eta}}(\theta \phi^{-1})
= N_{L^{nr}_{P,\eta}/F_{P,\eta}}(\theta)
[N_{L^{nr}_{\eta}/F_{\eta}}(\phi)]^{-1}$$ is a
$N^{\rm{th}}$ power in $F_{P,\eta}$.
Now $\lambda' = N_{L^{nr}_{\eta}/F_{\eta}}(\sqrt[e]{u\pi}^{i}\phi)$
has the desired property.

\end{proof}

\begin{lemma}
\label{at_least_one_unramified_case_surjectivity_eta_to_branch_multinorm}
Let $K$ be a complete discretely valued field
with residue field $\kappa$
and $F$ be
the function field of a curve over $K.$
Assume that $\kappa$ is algebraically closed.
Assume that $L_i/F, 1\leq i \leq m,$ are
cyclic extensions with $[L_i:F] = p$,
where $p$ is a prime number.
Assume that $p \neq \text{char}(\kappa)$.
Let $T_{L/F}$ denote the multinorm torus associated to
the extensions $L_i/F$. 
Let $\XX$, $X$ and $\PP$ as in Lemma \ref{approximating_norm_from_eta_to_branch_multinorm}. Let $\eta$ be the generic point of
one of the components of $X \setminus \mathcal{P}$ and
let $(P,\eta)$ be a branch.
Assume that for at least one $i$,
$L_{i,P,\eta} = L_i \otimes_F F_{P,\eta}$ is unramified
over $F_{P,\eta}$.
Then the natural map
$$T_{L/F}(F_{\eta}) \rightarrow T_{L/F}(F_{P,\eta})/R$$
is surjective.
\end{lemma}

\begin{proof}

We can assume that  all $L_{i,P,\eta}$
are fields. Otherwise, if $L_{i,P,\eta}$ is
not a field for some $i$, then $L_{i,P,\eta}$
is a product of $p$ copies of $F_{P,\eta}$.
Thus, in this case, we get that $T_{L/F}(F_{P,\eta})/R = \{1\}$
by Lemma \ref{torus_F_times_separable_algebra_is _R_trivial}.
Hence the conclusion of the lemma holds.\\

Let $(\mu_1, \dots,\mu_{m-1}, \mu_m)
\in T_{L/F}(F_{P,\eta}).$
Let $\lambda_i = N_{L_{i,P,\eta}/F_{P,\eta}}(\mu_i)$
for all $i$, $1 \leq i \leq m.$ Then
$\displaystyle \prod_{i=1}^{m} \lambda_i =1.$
Without loss of generality, we can assume that the
extension $L_{m,P,\eta}/F_{P,\eta}$ is unramified.
Let $\pi \in F_{\eta}$ be a parameter. Then
$\pi$ is also a parameter in $F_{P,\eta}$ and
$L_{m,P,\eta}$.
By Lemma \ref{approximating_norm_from_eta_to_branch_multinorm},
for $i$, $1 \leq i \leq m-1$, we can find $\mu'_i \in L_{i,\eta}$ with
such that
$\lambda_i [N_{L_{i,\eta}/F_{\eta}}(\mu'_i)]^{-1}
 = \alpha^{p}_i$ for some
 $\alpha_i \in F_{P,\eta}^{\times}.$
For $1 \leq i \leq m-1$, let $\lambda'_i = N_{L_{i,\eta}/F_{\eta}}(\mu'_i).$
 Since $\displaystyle \prod_{i=1}^{m-1} \lambda_i = (\lambda_m)^{-1}$
 and $\lambda_m = N_{L_{m,P,\eta}/F_{P,\eta}}(\mu_m)$,
$p$ divides $\displaystyle\text{val}_{\pi}(\prod_{i=1}^{m-1} \lambda_i)$
where $\text{val}_{\pi}$ denotes the valuation on
$F_{P,\eta}$ with parameter $\pi$.
Since, for $1 \leq i \leq m-1$, $\lambda_i {\lambda'}^{ -1}_i
 = \alpha^{p}_i$ with
 $\alpha_i \in F_{P,\eta}^{\times},$
 $p$ also divides $\displaystyle\text{val}_{\pi}(\prod_{i=1}^{m-1} {\lambda'}^{-1}_i)$.\\

Let $F(\eta)$ and $L(m,\eta)$ be
the residue field of
$F_{\eta}$ and $L_{m,\eta}$ respectively.
Since $\kappa$ is algebraically closed,
$F(\eta)$ is a $C_1$ field by \cite[Thm 6.2.8, p-143]{philippe2006central} and thus has cohomological dimension $\leq 1$ by \cite[Thm 6.2.3, p-143]{philippe2006central}.
Note that $L(m,\eta)/ F(\eta)$ is a degree $p$ separable extension, thus the norm map $N_{L(m,\eta)/ F(\eta)}$
is surjective by \cite[Thm 6.1.8, p-138]{philippe2006central}.
Since $p$  divides
$\displaystyle\text{val}_{\pi}(\prod_{i=1}^{m-1} {\lambda'}^{-1}_i)$,
we can find a
$\mu'_m \in L^{\times}_{m,\eta}$ with
$N_{L_{m,\eta}/F_{\eta}}(\mu'_m) = \prod_{i=1}^{m-1} {\lambda'}^{-1}$.
Then $(\mu'_1,\dots, \mu'_m) \in
 T_{L/F}(F_{\eta})$ and
 $N_{L_{i,P,\eta}/F_{P,\eta}}(\mu_i{\mu'}^{-1}_i)
 \in F^{\times p}_{P,\eta}$ for $1 \leq i \leq m$.
 Also, since the extensions
 $L_{i,P,\eta}/F_{P,\eta}$ are cyclic,
 $T_{L/F}(F_{P,\eta})/R =\{1\}$
 by Lemma \ref{R_trivial_elements_cyclic_norm_one_tori}.
 Hence, the result follows by Lemma \ref{surjective_for_general_extensions}.

 \end{proof}

\begin{lemma}
\label{surjectivity_eta_to_branch_multinorm}

Let $K$ be a complete discretely valued field
with residue field $\kappa$
and $F$ be
the function field of a curve over $K.$
Assume that $\kappa$ is algebraically closed.
Assume that $L_i/F, 1\leq i \leq m,$ are
quadratic extensions.
Assume that $\text{char}(\kappa) \neq 2$.
Let $T_{L/F}$ denote the multinorm torus associated to
the extensions $L_i/F$. 
Let $\XX$, $X$ and $\PP$ as in Lemma \ref{approximating_norm_from_eta_to_branch_multinorm}. Let $\eta$ be the generic point of
one of the components of $X \setminus \mathcal{P}$ and
let $(P,\eta)$ be a branch.\\

Then the natural map
$$T_{L/F}(F_{\eta}) \rightarrow T_{L/F}(F_{P,\eta})/R$$
is surjective.

\end{lemma}

\begin{proof}

We can assume that  all $L_{i,P,\eta}$
are fields. Otherwise, if $L_{i,P,\eta}$ is
not a field for some $i$, then $L_{i,P,\eta}$
is a product of two copies of $F_{P,\eta}$.
Thus, in this case, we get that $T_{L/F}(F_{P,\eta})/R = \{1\}$
by Lemma \ref{torus_F_times_separable_algebra_is _R_trivial}.
Hence the conclusion of the lemma holds.\\

\textbf{Case A} Let us assume that
at least one of the extensions
$L_{i,P,\eta}/F_{P,\eta}$ is
unramified. Then we have the result
by Lemma \ref{at_least_one_unramified_case_surjectivity_eta_to_branch_multinorm}.\\

\textbf{Case B} Now we assume that
all the extensions $L_{i,P,\eta}/F_{P,\eta}$
are ramified. 
The field $F_{P, \eta}$ is a complete discretely valued field with residue field $\kappa(\eta)_P$. Let $\pi$ be a parameter in $F_{P,\eta}$. Then by \cite[Lemma 2.4]{parimala2018local}, any ramified quadratic extension of $F_{P,\eta}$ is of the form $F_{P,\eta}(\sqrt{u\pi})$, where $u\in F_{P,\eta}$ is a unit. Note that $\kappa(\eta)_P$ is a complete discretely valued field with residue field an algebraically closed field. Let $\overline{u}$ be the residue of $u$ in $\kappa(\eta)_P$. Then $\overline{u}$ is a square in $\kappa(\eta)_P$ if and only if $\overline{u}$ is a unit in $\kappa(\eta)_P$. Thus, depending on whether $\overline{u}$ is a unit in $\kappa(\eta)_P$ or not, we conclude that there are exactly two non-isomorphic ramified quadratic extensions of $F_{P, \eta}$. 
In this case, by
Proposition \ref{only_nonisomorphic_fields_matter_multinorm_mod_R},
it is enough to consider the case when $m=2$
and $L_{1,P,\eta}/F_{P,\eta}$ and
$L_{2,P,\eta}/F_{P,\eta}$ are both ramified
quadratic extensions.
In this case, by Lemma \ref{both_ramified_ext_tori_mod_R_trivial},
we have $T_{L/F}(F_{P,\eta})/R =\{1\}.$

\end{proof}

\begin{lemma}
\label{branch_to_P_multinorm_general}

Let $K$ be a complete discretely valued field
with residue field $\kappa$
and $F$ be
the function field of a curve over $K.$ 
Assume that $L_i,$ $1\leq i \leq m,$ are
finite Galois extensions of degree $n_i$, where $n_i$ is a
natural number not divisible by $\text{char}(\kappa)$. 
Let $\XX$ be a regular proper model of $F$
such that the union of $ram_{\XX}(L_i/F)$ and
the reduced special fibre $X$ is a union of
regular curves with normal crossings.
Let $T_{L/F}$ denote the multinorm torus associated to
the extensions $L_i/F$.
Assume that the natural maps
$$ T_{L_i/F}(F_{P}) \rightarrow T_{L_i/F}(F_{P,U})/R $$
are surjective for $1 \leq i \leq m$.
Let $(P,U)$ be a branch.\\

Then the natural map
$$T_{L/F}(F_{P}) \rightarrow T_{L/F}(F_{P,U})/R$$
is surjective.

\end{lemma}

\begin{proof}

Let $\mu_i \in L_{i,P,U}$ for $1 \leq i \leq m$
with $N_{L_{i,P,U}/F_{P,U}}(\mu_i) = \lambda_i$
such that $\displaystyle \prod_{i=1}^{m} \lambda_i = 1.$
Let $\displaystyle n \colonequals
 \mathsf{lcm} (n_i \mid 1 \leq i \leq m)$.
By \cite[Lemma 4.7]{mishra2021local}, we can write $\lambda_i =
\lambda'_i \cdot (\alpha_i)^{n}$ for
some $\alpha_i \in F_{P,U},$
where $\lambda'_i = u_i \pi^{s_i}_1\pi^{t_i}_2$
for some $u_i \in \hat{R}^{\times}_{P}$ and integers $s_i,t_i$.
By \cite[Theorem 4.9]{mishra2021local}, we can choose $\mu'_i \in L_{i,P}$
with $N_{L_{i,P}/F_{P}}(\mu'_i) = \lambda'_i$
for $1 \leq i \leq m-1.$ Similarly, we choose
$\mu'_m \in L_{m,P}$ with
$\displaystyle N_{L_{m,P}/F_P}(\mu'_m) =
(\prod_{i=1}^{m-1}\lambda'_{i})^{-1}$.
Then $(\mu'_1, \mu'_2, \dots , \mu'_m)
\in T_{L/F}(F_{P})$. Now the result follows from
Lemma \ref{surjective_for_general_extensions}.

\end{proof}

\begin{cor}
\label{surjectivity_P_to_branch_multinorm}

Let $K$ be a complete discretely valued field
with residue field $\kappa$
and $F$ be
the function field of a curve over $K$.
Assume that $\kappa$ is algebraically closed.
Assume that $L_i/F, 1\leq i \leq m,$ are
finite Galois extensions of degree $n_i$, where $n_i$ is a
natural number not divisible by $\text{char}(\kappa)$.
Let $T_{L/F}$ denote the multinorm torus associated to
the extensions $L_i/F$.
Let $\XX$ be a regular proper model of $F$
such that the union of $ram_{\XX}(L_i/F)$ and
the reduced special fibre $X$ is a union of
regular curves with normal crossings.
Let $(P,U)$ be a branch.\\

Then the natural map
$$T_{L/F}(F_{P}) \rightarrow T_{L/F}(F_{P,U})/R$$
is surjective.

\end{cor}

\begin{proof}

Since $\kappa$ is algebraically closed, 
by \cite[Lemma 6.1]{mishra2021local} we have$\colon$
$T_{L_{i}/F}(F_{P,U})/R = \langle \rho_{n_i} \rangle $
for all $i; 1 \leq i \leq m$, where $\rho_{n_i}$
is a primitive $n_i^{\rm{th}}$ root of unity
in $F.$
Thus the natural maps
$$ T_{L_i/F}(F) \rightarrow T_{L_i/F}(F_{P,U})/R $$
are surjective for $1 \leq i \leq m$ since
$\rho_{n_i} \in  T_{L_i/F}(F)$. Thus, the
result follows from Lemma \ref{branch_to_P_multinorm_general}.

\end{proof}

\begin{theorem}
\label{Sha_dvr_vanishes_for_multinorm}
Let $K$ be a complete discretely valued field
with residue field $\kappa$ algebraically closed.
Let $F$ be the function field of a curve over $K$. 
Let $\XX$ be a regular proper model of $F$
such that the reduced special fibre $X$ is a union of
regular curves with normal crossings.
Let $L_i/F, 1\leq i \leq m,$ be
quadratic extensions.
Assume that $\text{char}(\kappa) \neq 2$.
Let $\displaystyle L = \prod_{i=1}^{m} L_i$.
If the graph associated to $\XX$ is a tree
then the local-global principle holds for the multinorm torus $T_{L/F}$
with respect to discrete valuations.
\end{theorem}

\begin{proof} 
By \cite[p-193]{lipman1975introduction},
there exists a  sequence of blow-ups $\XX_0 \to \XX$
centered at closed points of $\XX$
 such that the union of $ram_{\XX}(L/F)$
and the reduced special fibre $X_0$ of $\XX_0$
is a union of regular curves with normal crossings.
Now by Theorem \ref{union_of_Sha_X_equals_Sha_dvr_multinorm},
it is enough to show that for any blowup $\mathscr{Y}$
of $\XX_0$, $\Sh_{\mathscr{Y},Y}(F, T_{L/F}) = 0$, where $Y$ is the closed fiber of $\mathscr{Y}$. This follows from Proposition \ref{eta_P_implies_ShaX_zero_multinorm}, whose surjectivity hypotheses respectively follows from Lemma \ref{F_to_F_Henselian} (noting that the valuation ring of $F_{\eta, h}$ is excellent since $K$ is complete discretely valued field; see \cite[Scholie 7.8.3(ii)]{grothendieck1965elements}), Lemma \ref{surjectivity_eta_to_branch_multinorm}, and Corollary \ref{surjectivity_P_to_branch_multinorm}.
\end{proof}

The above statement can be generalized as follows$\colon$\\

\begin{theorem}
\label{Sha_dvr_vanishes_for_multinorm_quadratic_general}

Let $K$ be a complete discretely  valued field with residue field $\kappa$.
Assume that $\kappa$ is algebraically closed. 
Let $F$ be function field of a curve over $K$. 
Let $\XX$ be a regular proper model of $F$
such the reduced special fibre $X$ is a union of
regular curves with normal crossings.
Let $L_0$ be a quadratic extension of $F$.
For $i= 1, 2, \dots, m$, let $L_i$ be a Galois extension of
$F$ of degree $n_i$ with Galois group $G_i$.
Assume that for $i$, $1\leq i \leq m$, $4$ does not divide $n_i$.
Assume that $2$ and $n_i$ are not divisible by \text{char}$(\kappa)$.
Let $\displaystyle L = \prod^{m}_{i=0} L_i.$
If the graph associated to $\XX$ is a tree then the local-global principle holds for $T_{L/F}.$

\end{theorem}

\begin{proof} If $2$ does not divide $n_i$ for some $i$ then local-global principle holds by Observation \ref{coprime_degrees}.
So we can assume that $n_i = 2 r_i$ for
some odd positive integers $r_i$ for $i$, $1 \leq i \leq m$. Then by Lemma \ref{2m}, there exists an index two normal subgroup $H_i$ of $G_i$ for all $i$. \\

Let $L'_{i}$ be the fixed fields of $H_i$.
Assume that $\alpha \in F^{\times}$ is a norm from
$L \otimes_F F_{\nu}$ to $F_{\nu}$ for all $\nu \in \Omega_F$.
Then $\alpha$ is also a norm from
$(L_0 \times L'_{1} \times L'_2  \times \cdots \times L'_m) \otimes_F F_{\nu} $ to $F_{\nu}$
for all $\nu \in \Omega_F$.
By Theorem \ref{Sha_dvr_vanishes_for_multinorm},
$\alpha$ is a norm from
$L_0 \times L'_1 \times L'_2 \times \cdots L'_m$ to $F$.
Let $r = \mathsf{lcm}(r_1, r_2, \dots, r_m)$ and $s_i=r/r_i$ for $i;$ $1\leq i\leq m$.
We have $\alpha = N_{L_0/F}(a_0) \displaystyle \prod^{m}_{i=1} N_{L'_i/F}(a_i)$
for some $a_0\in L^{\times}$ and $a_i \in L^{'\times}_i;$ $1\leq i\leq m$.
Then $\alpha^r= [N_{L_0/F}(a_0) \displaystyle \prod^{m}_{i=1} N_{L'_i/F}(a_i)]^r= N_{L_0/F}(a^r_0) \displaystyle \prod^{m}_{i=1}N_{L_i/F}(a^{s_i}_i).$
Thus $\alpha^{r}$ is a norm from $L$ to $F$.\\

Now let $E_i$ be an order two subgroup of $G_i$ for $i$, $1 \leq i \leq m$.
Let $L^{\prime \prime}_{i}$ be the fixed fields of $E_i$.
Assume that $\alpha \in F^{\times}$ is a norm from
$L \otimes_F F_{\nu}$ to $F_{\nu}$ for all $\nu \in \Omega_F$. Since $[L^{\prime \prime}_{i}:F]=r_i$ are odd, by Observation \ref{coprime_degrees} we get that $\alpha$ is a norm from
$F \times L^{\prime \prime}_{1} \times L^{\prime \prime}_2  \times \cdots \times L^{\prime \prime}_m$
to $F$.
Hence $\alpha^2$ is a norm from $L$ to $F$. Since $r$ and $2$ are coprime and
$\alpha^{r}, \alpha^2$ are products of norms
from $L$ to $F$, we conclude that $\alpha$ is a norm from $L$ to $F$.
\end{proof} 

\textbf{Acknowledgements}:
The author thanks Prof. Suresh Venapally for various helpful discussions during this work and Prof. Parimala Raman for constant encouragement and her interest in this work. The author is very thankful to the anonymous referees for their comments and suggestions, due to which the exposition has improved significantly. The author thanks Prof. Preeti Raman and IIT Bombay for their support.

\bibliography{Multinorms_25May2023}

\bibliographystyle{amsalpha}

\end{document}